\documentclass[11pt]{article}
	
	\newcommand{\blind}{0}
    \makeatletter
    \renewcommand\section{\@startsection {section}{1}{\z@}%
                                       {-3.5ex \@plus -1ex \@minus -.2ex}%
                                       {2.3ex \@plus.2ex}%
                                       {\normalfont\fontfamily{phv}\fontsize{16}{19}\bfseries}}
    \renewcommand\subsection{\@startsection{subsection}{2}{\z@}%
                                         {-3.25ex\@plus -1ex \@minus -.2ex}%
                                         {1.5ex \@plus .2ex}%
                                         {\normalfont\fontfamily{phv}\fontsize{14}{17}\bfseries}}
    \renewcommand\subsubsection{\@startsection{subsubsection}{3}{\z@}%
                                        {-3.25ex\@plus -1ex \@minus -.2ex}%
                                         {1.5ex \@plus .2ex}%
                                         {\normalfont\normalsize\fontfamily{phv}\fontsize{14}{17}\selectfont}}
    \makeatother
	
	\usepackage{amsmath}
	\usepackage{graphicx}
	\usepackage{enumerate}
	\usepackage{natbib} 
	\usepackage{url} 
    \usepackage{xcolor}
	
    \newcommand{\TD}{\nabla}
    \usepackage{bm}
    \usepackage{multicol}
    \usepackage{lipsum}
    \usepackage{mwe}
    \usepackage{graphicx}
    \usepackage{subfig}
    \usepackage{amssymb}
    \usepackage{xcolor}
    \usepackage{amsthm}
    \usepackage{amsfonts}
    \usepackage{graphicx}
    \usepackage{epstopdf}
    \usepackage{algorithm,algcompatible}

    \newtheorem{theorem}{Theorem}[section]
    \newtheorem{lemma}[theorem]{Lemma}
    \newtheorem{corollary}[theorem]{Corollary}
    \newtheorem{proposition}[theorem]{Proposition}
    \newtheorem{assumption}{Assumption}

    \newtheorem{definition}[theorem]{Definition}
    \newtheorem{example}[theorem]{Example}

    \newtheorem{remark}{Remark}

    \DeclareMathOperator*{\argmin}{\arg\!\min}

    \newcommand{\real}{{\rm{I\hspace{-.75mm}R}}}

    \renewcommand*{~}{\relax\ifmmode\sim\else\nobreakspace{}\fi}

    \newcommand{\Ebar}{\bar{E}}
    \newcommand{\Fbar}{\bar{F}}

    \newcommand{\Nhat}{\hat{N}}

    \newcommand{\Rhat}{\hat{R}}

    \newcommand{\rhohat}{\hat{\rho}}
    \newcommand{\sigmahat}{\hat{\sigma}}

    \newcommand{\Ntilde}{\widetilde{N}}
    
    \newcommand{\Rtilde}{\widetilde{R}}
    
    \newcommand{\BFe}{\bm{e}}
    
    \newcommand{\BFg}{\bm{g}}

    \newcommand{\BFs}{\bm{s}}

    \newcommand{\BFx}{\bm{x}}
    \newcommand{\BFy}{\bm{y}}
    
    \newcommand{\BFE}{\bm{E}}
    \newcommand{\BFF}{\bm{F}}
    \newcommand{\BFG}{\bm{G}}
    \newcommand{\BFH}{\bm{H}}

    \newcommand{\BFX}{\bm{X}}
    \newcommand{\BFS}{\bm{S}}
    \newcommand{\BFY}{\bm{Y}}
    
    \newcommand{\BFalpha}{\bm{\alpha}}
    \newcommand{\BFbeta}{\bm{\beta}}

    \newcommand{\BFvarepsilon}{{\bm{\varepsilon}}}
    \newcommand{\BFtheta}{\bm{\theta}}

    \newcommand{\BFXhat}{\hat{\BFX}}
    
    \newcommand{\BFXtilde}{\widetilde{\BFX}}

    \newcommand{\sfH}{\mathsf{H}}
    \newcommand{\sfM}{\mathsf{M}}

    \newcommand{\mcA}{\mathcal{A}}
    \newcommand{\mcB}{\mathcal{B}}
    
    \newcommand{\mcD}{\mathcal{D}}
    \newcommand{\mcE}{\mathcal{E}}
    \newcommand{\mcF}{\mathcal{F}}
    
    \newcommand{\mcH}{\mathcal{H}}

    \newcommand{\mcK}{\mathcal{K}}

    \newcommand{\mcN}{\mathcal{N}}
    \newcommand{\mcO}{\mathcal{O}}
    \newcommand{\mcP}{\mathcal{P}}
    
    \newcommand{\mcR}{\mathcal{R}}
    \newcommand{\mcS}{\mathcal{S}}

    \newcommand{\mcV}{\mathcal{V}}
    
    \newcommand{\mcX}{\mathcal{X}}
    \newcommand{\mcY}{\mathcal{Y}}

    \newcommand{\mbE}{\mathbb{E}}

    \newcommand{\mbN}{\mathbb{N}}
    
    \newcommand{\mbP}{\mathbb{P}}
    
    \newcommand{\mbR}{\mathbb{R}}

    \newcommand{\mbV}{\mathbb{V}}

\begin{document}
		
		\def\spacingset#1{\renewcommand{\baselinestretch}%
			{#1}\small\normalsize} \spacingset{1}
		
		\if0\blind
		{
			\title{\bf 
   Iteration Complexity and Finite-Time Efficiency of Adaptive Sampling Trust-Region Methods for Stochastic Derivative-Free Optimization}
			\author{Yunsoo Ha $^a$ and Sara Shashaani $^a$ \\
			$^a$ Edward P. Fitts Department of Industrial and System Engineering, \\North Carolina State University, Raleigh, NC 27695, USA}
			\date{}
			\maketitle
		} \fi
		
		\if1\blind
		{

            \title{\bf Iteration Complexity and Finite-Time Efficiency of Adaptive Sampling Trust-Region Methods for Stochastic Derivative-Free Optimization}
			\author{Author information is purposely removed for double-blind review}
			
\bigskip
			\bigskip
			\bigskip
			\begin{center}
				{\LARGE\bf Iteration Complexity and Finite-Time Efficiency of Adaptive Sampling Trust-Region Methods for Stochastic Derivative-Free Optimization}
			\end{center}
			\medskip
		} \fi
		\bigskip
		
	\begin{abstract}
Adaptive sampling with interpolation-based trust regions or ASTRO-DF is a successful algorithm for stochastic derivative-free optimization with an easy-to-understand-and-implement concept that guarantees almost sure convergence to a first-order critical point.  To reduce its dependence on the problem dimension, we present local models with diagonal Hessians constructed on interpolation points based on a coordinate basis. We also leverage the interpolation points in a direct search manner whenever possible to boost ASTRO-DF's performance in a finite time. We prove that the algorithm has a canonical iteration complexity of $\mcO(\epsilon^{-2})$ almost surely, which is the first guarantee of its kind without placing assumptions on the quality of function estimates or model quality or independence between them. Numerical experimentation reveals the computational advantage of ASTRO-DF with coordinate direct search due to saving and better steps in the early iterations of the search.
	\end{abstract}
			
	\noindent%

	\spacingset{1.5} 


\section{Introduction}\label{sec:contribution}
We consider unconstrained simulation optimization (SO) problems of the form 
\begin{equation}
    \min_{\bm{x}\in \mbR^d} f(\bm{x}) := \mbE[F(\bm{x},\xi)] = \int_\Xi F(\bm{x},\xi)dP(\xi),\label{eq:problem}
\end{equation}
where $f:\mbR^d \rightarrow \mbR$ is smooth (but nonconvex) and bounded below, and $F:\mbR^d\times\Xi \rightarrow \mbR$ is the stochastic function value  defined on the probability space $(\Xi,\mcF,P)$.  
In particular, consider $f(\BFx)$ that is only observable with noise by a Monte Carlo simulation, which generates the random variable $F(\BFx, \xi)$. Hence, the estimator of $f(\BFx)$ via sample average approximation and the estimated variance of $F(\BFx, \xi)$ with $n$ runs of the simulation can be obtained by
$$\bar{F}(\BFx,n)=n^{-1}\sum_{i=1}^n F(\BFx,\xi_i),\text{ and }\sigmahat_F^2(\BFx,n)=(n-1)^{-1}\sum_{j=1}^n\left(F(\BFx,\xi_i)-\bar{F}(\BFx,n)\right)^2.$$ Furthermore, assume that derivative information is not directly available from the Monte Carlo Simulation. Hence Problem \eqref{eq:problem} becomes a stochastic derivative-free optimization (SDFO) problem. A SDFO algorithm produces, in one run, $\{\BFX_k\}$ -- a sequence of stochastic incumbent solutions for iterations $k\in\mbN$. This sequence can be viewed as a stochastic process defined on a filtered probability space $(\Omega,\mcF_k,\mbP)$, where $\mcF_k$ denotes the $\sigma$-algebra increasing in $k$. The goal is designing an efficient SDFO algorithm to reach an $\epsilon$-stationary point, that is a point in the set $\{\BFx:\ \|\nabla f(\BFx)\|\leq \epsilon\}$.

The number of function evaluations $n$ at a solution $\BFX_k$ generated by the SDFO algorithm must guarantee estimation accuracy of $\mathbb{P}\{|\bar{E}(\BFX_k,n)|>\varepsilon_k\}\le \alpha_k$, with the required accuracy threshold $\varepsilon_k>0$ and exceedance probability $\alpha_k\in(0,1)$, where $\bar{E}(\BFX_k,n):=\bar{F}(\BFX_k,n)-f(\BFX_k)$ denotes the estimation error. 
An insufficient $n$ during the optimization will threaten the convergence of the SDFO algorithm while an unnecessarily large $n$ at every design point leads to exhausting the simulation budget before reaching a reasonably good solution. Therefore, choosing just enough $n$ for each visited solution not only secures strong consistency but also reinforces efficiency in finite-time and asymptotically.

There is emerging evidence \citep{gratton2018complexity,maggiar2018derivative} that stochastic trust-region methods are effective at solving nonconvex SDFO, mainly due to their natural ability to self-tune step sizes and facility for curvature estimation. ASTRO-DF~\citep{shashaani2016astro,Sara2018ASTRO,ha2021improved} is a class of trust-region optimization algorithms suited for SDFO that deals with the challenge of choosing the right sample size using an adaptive sampling strategy. For a visited solution $\BFX_k$, using a random sample size $N(\BFX_k)$ that adapts to $\varepsilon_k$ and $\alpha_k$ at iteration $k$ can lead to convergence to a first-order stationary point almost surely. Given its versatile mechanics, guarantees in theory with little requirements, and promising results from straightforward implementation, ASTRO-DF is a favorable choice to tackle high-dimensional SDFO using new features. Therefore, this paper investigates the establishment of theoretical and finite-time efficiency properties of ASTRO-DF. 

\paragraph{List of Contributions} 
In ASTRO-DF, the algorithm's progress relies on a local model that, due to the derivative-free setting, is constructed on function value estimates of points adjacent to the incumbent solution within a trust region of size $\Delta_k$. The mechanism of trust-region optimization dictates that $\Delta_k$ eventually drops to 0 to drive the algorithm's convergence. Approximately minimizing the local model within the trust region (a.k.a., subproblem) provides a candidate for the next incumbent solution. The decreasing accuracy threshold for each point in iteration $k$ is identified as $\varepsilon_k=\mcO(\Delta_k^2)$ and the exceedance probability $\alpha_k$ drops to 0 in $k$ due to the $\mbP\{N(\BFX_k)\to\infty$ as $k\to\infty\}=1$  assurance. The sample size relates to the step size (controlled by the trust-region radius) with $N(\BFX_k)=\mcO(\Delta_k^{-4})$. Despite tremendous analytical gains, ASTRO-DF has some practical disadvantages:
\begin{enumerate}
    \item $\mcO{(d^2)}$ number of points needed to build a quadratic model to capture the local curvature information is too costly, especially considering that an increasing number of oracle runs is required in each of these points. Furthermore, the random choice of the design points causes significant inefficiency: to obtain a good model quality, the design points need to be ``well-poised" (or well spread) in the trust-region with the linear algebra cost of $\mcO(d^6)$ for a quadratic model; see Chapter 6 in \cite{katya:DFObook}.
    \item Even though ASTRO-DF converges to a stationary point almost surely, the probability of having a successful iteration, which is the probability of finding a better solution at each iteration, can start off very small, even though it tends to 1 as $k\to\infty$. Rejecting the candidate solution suggested by the model frequently reduces the trust-region radius $\Delta_k$ too quickly, which in turn can lead to running out of simulation budget before finding very good solutions. In other words, full reliance on the local model minimizer can slow the algorithm's progress. 
\end{enumerate}

To address these challenges, we propose two refinements to ASTRO-DF. The first refinement is to use a fixed geometry for the design set based on the coordinate basis and constructing a model with a diagonal Hessian~\citep{coope2020gradient}. Such a fixed geometry was  recently shown to be an optimal design \citep{tom2023optimalpoised}. Selecting the points with this design, besides eliminating a source of randomness inside the algorithm, also drops the linear algebra cost of certifying that the design is well-poised and the row operations needed to solve the system of equations in the interpolation to construct the model. More importantly, if we refer to the stochastic error as $E$, the upper bound for the model gradient error will be reduced to $\mcO(\Delta_k^2 + E/\Delta_k)$ that is a central-difference like error bound, from $\mcO(\Delta_k + E/\Delta_k)$ that is a forward-difference like error bound in the original algorithm. 

The second refinement is to leverage the design points in a direct search manner. This refinement is motivated by observing that in each iteration of ASTRO-DF, the estimated function value at the candidate solution recommended by the local model can be worse than that of the points used to construct the model. Despite such possibilities, the original setup deems the iteration unsuccessful, ignoring the readily identified points that obtain improvement, and starts over to build a model in a smaller neighborhood. In addition to smaller steps, this rejection will significantly increase oracle runs, incurring a considerable waste of simulation budget. A direct search strategy considers selecting design points whose estimated function values provide sufficient reduction in the function value estimates as the next incumbent, enabling progress without extra cost and maintaining stability in step size and sample size in the early iterations. 

\begin{figure} [htp]
\centering
\includegraphics[width=0.65\columnwidth]{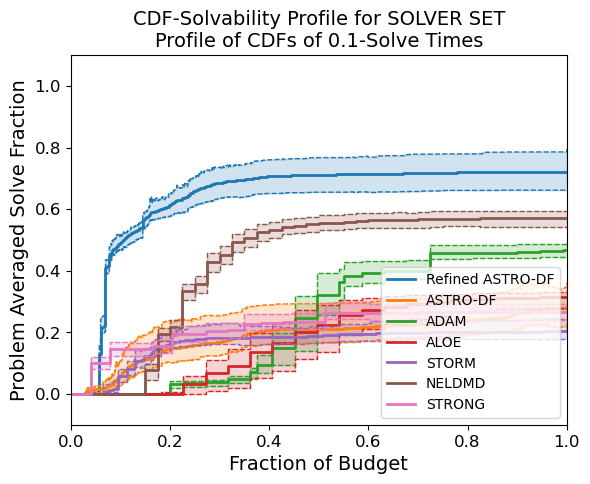}
\caption{Fraction of 60 problems from SimOpt library~\citep{simoptgithub} solved to $0.1$-optimality with 95\% confidence intervals from 20 runs of each algorithm  shows a clear advantage in finite-time performance of ASTRO-DF with coordiate direct search.} \label{fig:entire}
\end{figure}

\begin{figure} [htp]
\centering
\subfloat[20-dim noisy Rosenbrock]{%
\resizebox*{7cm}{!}{\includegraphics{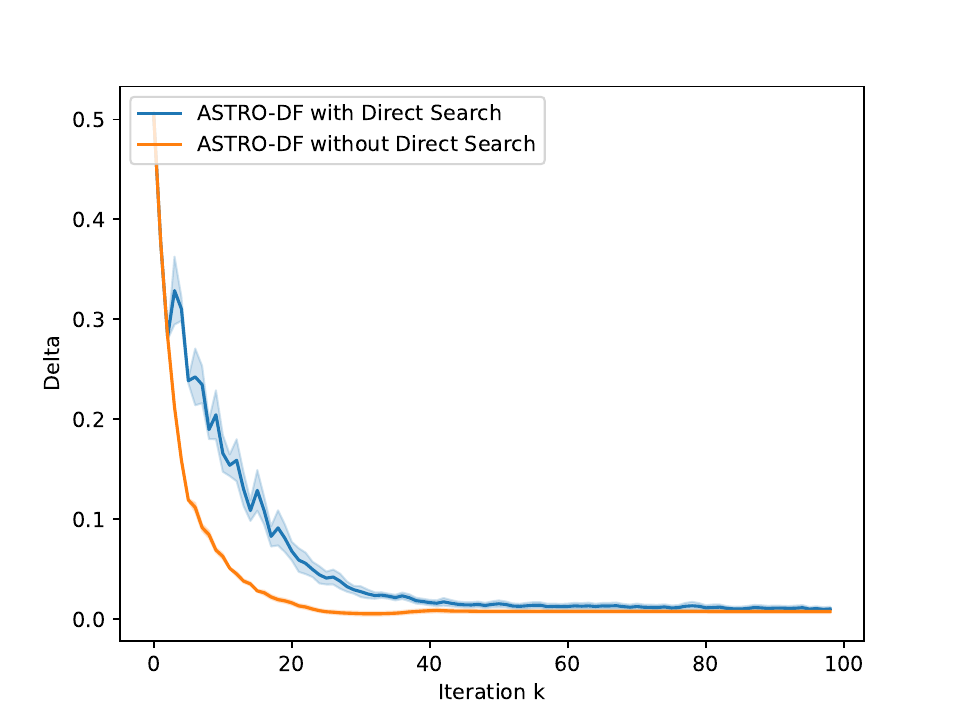}}\label{fig:delta1}}\hspace{5pt}
\subfloat[13-dim Stochastic Activity Network]{%
\resizebox*{7cm}{!}{\includegraphics{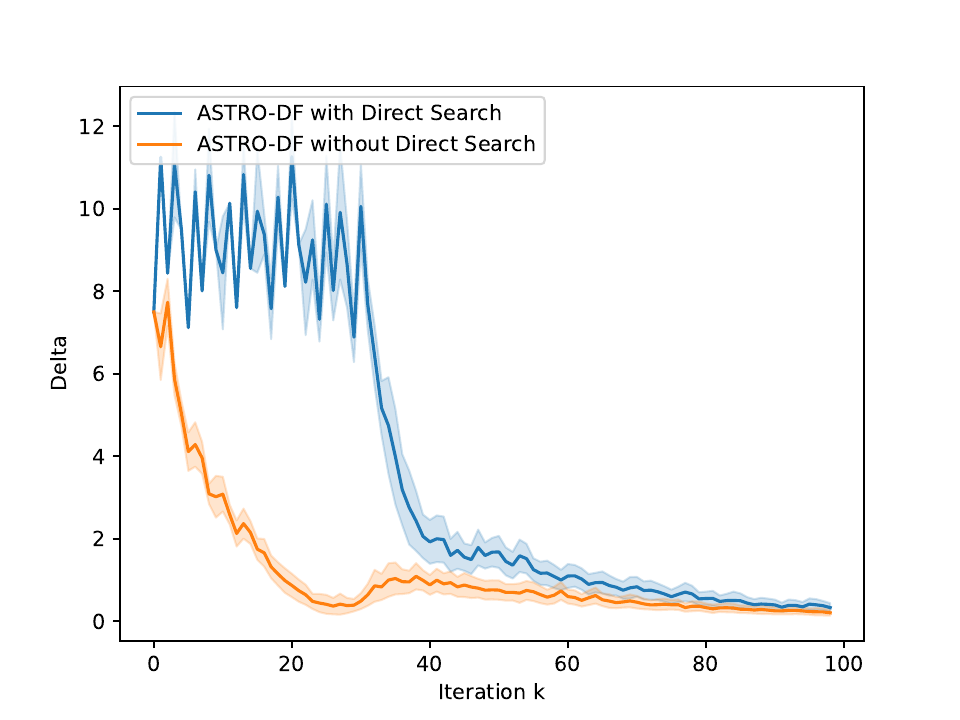}}\label{fig:delta2}}
\caption{Slower rate of trust-region radius decay in the early iterations of the search as a result of using direct search allows for larger step size and hence better progress per iteration early on. The mean and 95\% confidence interval of trajectory of $\Delta_k$ exhibits this behavior for two problems.} \label{fig:delta}
\end{figure}

After a review of existing research related to SDFO in Section \ref{sec:litreview}, we detail the coordinate direct search features of ASTRO-DF in Section \ref{sec:algorithm}. These refinements significantly affect the finite-time performance, as summarized in Figure~\ref{fig:entire} that compares the ASTRO-DF solver that uses coordinate direct search with the state-of-the-art SDFO algorithms. We will also show in Section \ref{sec:complexity} that the refinements made to affect finite-time progress are not harmful to the strong consistency. More importantly, we prove the canonical iteration complexity, that is, $\mcO(\epsilon^{-2})$ for the number of iterations to reach $\epsilon$-optimality in expectation for ASTRO-DF. This result is stronger than the best existing result proven for an algorithm that uses random models by~\cite{chen2018storm} since it relaxes the requirements on model quality and function estimates and their independence from one another and the probability of success in the algorithm~\citep{blanchet2019convergence}. Although difficult to prove analytically, we discuss that the direct coordinate search improves the iteration complexity in the constant terms. Section \ref{sec:numerical} explores extensive numerical experiments on two classes of convex and nonconvex moderate dimensional problems that are considered expensive to solve in the derivative-free setting. In particular, we observe (as evidenced by Figure~\ref{fig:delta}) the intuitive boost in finite-time performance of the algorithm by using direct search since this feature increases the success probability in each iteration (probability of finding a new incumbent) prevents the trust-region size from becoming too small in the early iterations of optimization. Success probability will eventually tend to that of the original ASTRO-DF as $k\to\infty$ since the models become progressively better approximations of the function and more likely to suggest better solutions than the points identified with the fixed geometry (coordinate basis). We will summarize our findings and future needs in Section \ref{sec:conclusion}.

\section{Preliminaries} \label{sec:prelim}
In this section, we review the existing literature for SDFO and introduce notations and essential definitions.
\subsection{Literature Review}
\label{sec:litreview}
In the artificial intelligence era, derivative-free optimization has received much attention for allowing the user to specify the objective function involved in non-explicit forms. As a result, derivative-free optimization has a wide range of applications such as hyper-parameter tuning \citep{ghanbari2017blackbox, ruan2020learning}, reinforcement learning \citep{ choromanski2018rl, fazel2018rl, flaxman2004bandit, salimans2017evolution}, simulation-based optimization \citep{Chang2013STRONG}, and quantum computing \citep{menickelly2022quantum}. An essential characteristic of SDFO is that the function evaluations are only accessible via a black-box simulation with stochastic noise. Running this noisy simulation can be expensive. Hence, one of persistent aims in this field is to improve the efficiency of the algorithms \citep{jin2021probability, paquette2020ls, Sara2018ASTRO}.

Efficiency guarantees of an algorithm for SDFO are secondary to the convergence guarantees that imply the iterates produced by the algorithm stabilize towards a stationary point in some probabilistic sense. An adaptive sampling strategy, which adapts the number of required function evaluations at each design point to the inferred closeness to optimality, leads to almost sure convergence to a stationary point. The intuition behind adaptive sampling is that more accurate estimates are needed for near-optimal solutions to guarantee we can identify better solutions. 

In the existing literature, the efficiency of SDFO algorithms is typically evaluated through a complexity analysis that assesses the costs required to achieve a near-optimal solution. The costs can take different forms, such as the number of iterations or function evaluations, known as ``iteration complexity" and ``work complexity," respectively. Work complexity is a more appropriate metric for gauging the computational load of SDFO algorithms, as the number of function evaluations could vary per iteration for the almost sure convergence. The existing body of work that analyzes SDFO algorithms has little known work complexity guarantees but has focused on attaining the canonical rate of expected iteration complexity $\mbE[T_\epsilon]=\mcO(\epsilon^{-2})$, where $T_\epsilon:=\{k\in\mbN:\|\TD f(\BFX_k)\|\leq\epsilon\}$ is the first iteration reaching $\epsilon$-optimality.
The complexity of SDFO problems proliferates with the problem dimension due to the necessity of the derivative estimator using only expensive, noisy simulations. The method for the derivative approximation and the sample size of each design point are the key factors in the complexity and finite-time performance. The two suggested refinements for ASTRO-DF here aim for more efficient solutions for higher dimensional derivative-free problems.

Recall, the SDFO needs function estimates, and the accuracy of the estimates is important for convergence. Hence, we divide the existing literature based on assumptions on the stochastic error. Let $E(\BFx)$ represent the stochastic error at $\BFx \in \mbR^{d}$, whose expectation may or may not be $0$, hence allowing bias estimates in the algorithms. 

\paragraph{uniformly bounded $|E(\BFx)|$:} Several methods for approximating gradients, such as a finite difference and a gaussian smoothing, are analyzed with bounds for the number of design points to obtain sufficiently accurate estimates, when stochastic error has a uniform constant upper bound \citep{berahas2021gradientest}. $\mcO(\epsilon^{-2})$ is obtained for expected iteration complexity in the generic line search that converges to an optimal neighborhood almost surely~\citep{berahas2021ls}. 

\paragraph{one-sided sub-exponential $E(\BFx)$:}  Given a fixed sample size $n$, $\mbE\left[E(\BFx)\right]=0$ does not hold because $\mbE[|E(\BFx)|]=0$ is not guaranteed. The adaptive line search with oracle estimations algorithm (ALOE)~\citep{jin2021probability} in this setting  converges to a neighborhood of stationarity with high probability tail bound for the iteration complexity  $\mbP\{T_\epsilon \le t\} \ge 1-\mcO(e^{-t})$. The trust-region method with this assumption on the stochastic errors has recently been developed~\citep{cao2022trhighp} exhibiting the same iteration complexity. 

\paragraph{zero-mean $E(\BFx)$ with a finite variance:} The trust region with random models algorithm (STORM)~\citep{chen2018storm} using this assumption converges to a stationary point almost surely with the expected iteration complexity $\mcO(\epsilon^{-2})$ \citep{blanchet2019convergence}. The convergence and complexity analysis for STORM relies on sufficiently replicating the function value at any given point $\BFx$ such that given $\varepsilon > 0$, $\mbP\{|\Ebar(\BFx,N(\BFx))|>\varepsilon\} \le \alpha$ for small enough $\alpha$. Moreover, in order to meet the condition for any $k \in \mbN$, STORM necessitates an extra assumption, namely that the function estimates and local models are independent. This implies that the information obtained in previous iterations cannot be reused in the current iteration, resulting in inefficiencies.

\paragraph{zero-mean sub-exponential $E(\BFx)$ with bounded $v$-th moment for $v\leq 8$:} ASTRO-DF using this assumption converges to a stationary point almost surely, and it requires $\mbP\{|\Ebar(\BFx,N(\BFx))|>\varepsilon\} \le \alpha_k$, where $\{\alpha_k\}$ is monotone decreasing and $\alpha_k\to0$ as $k\to\infty$. 

\paragraph{bounded variance of $E(\BFx)$:}  A backtracking Armijo line search method \citep{paquette2020ls} achieves $\mcO(\epsilon^{-2})$ as the expected iteration complexity if a variance of $E$ converges to 0 as $k \rightarrow \infty$.

\subsection{Notation}
We use bold font for vectors; $\BFx=(x_1,x_2,\cdots,x_d)\in\mbR^d$ denotes a $d$-dimensional vector of real numbers. Let $\BFe^{i} \in \mbR^d$ for $i=1,\dots,d$ denote the standard unit basis vectors in $\mbR^d$. We use calligraphic fonts for sets and sans serif fonts for matrices. $\mcB(\BFx;\Delta)=\{\BFx+\BFs\in\mbR^d:\|\BFs\|_2\leq\Delta\}$ denotes the closed ball of radius $\Delta>0$ with center $\BFx$. For a sequence of sets $\{\mcA_n\}$, the set $\{\mcA_n \ \text{i.o.}\}$ denotes $\limsup_{n} \mcA_n$. 
We say $f(\BFx)=\mcO(g(\BFx))$ if there exist positive numbers $\varepsilon$ and $M$ such that $|f(\BFx)|\le Mg(\BFx)$ for all $\BFx$ with $0 < |\BFx| < \varepsilon$. 
We use $a\land b:=\min\{a,b\}$.
We use capital letters for random scalars and vectors. For a sequence of random vectors $\{\BFX_k\},\ k\in\mbN$, $\BFX_k \xrightarrow{w.p.1}\BFX$ denotes convergence almost surely. The term ``iid", abbreviates independent and identically distributed. 

\subsection{Definitions} Next, we introduce several definitions in all of which, we start with a center point $\BFx$ (later, the incumbent solution $\BFX_k$) and the trust-region radius $\Delta$ (later, $\Delta_k$ that changes with $k$).
\begin{definition}(stochastic polynomial interpolation models).
    Given $\BFx\in\mbR^d$ and $\Delta>0$, let $\Phi(\BFx)=(\phi_0(\BFx),\phi_1(\BFx), \dots, \phi_q(\BFx))$ be a polynomial basis on $\mbR^d$. With $p=q$ and the design set $\mcX:=\{\BFx^{0}, \BFx^{1}, \dots , \BFx^{p}\}\subset \mcB(\BFx;\Delta)$, we can find $\BFbeta = (\beta_0,\beta_1, \dots, \beta_p)$ such that $\sfM(\Phi, \mcX) \BFbeta = \Fbar(\mcX,N(\mcX)),$
    where \\ 
    \begin{equation*}
        \sfM(\Phi, \mcX) = 
        \begin{bmatrix}
        \phi_1(\BFx^{0}) & \phi_2(\BFx^{0}) & \cdots & \phi_q(\BFx^{0})  \\
        \phi_1(\BFx^{1}) & \phi_2(\BFx^{1}) & \cdots & \phi_q(\BFx^{1})  \\
        \vdots & \vdots & \vdots & \vdots \\
        \phi_1(\BFx^{p}) & \phi_2(\BFx^{p}) & \cdots & \phi_q(\BFx^{p})  \\
        \end{bmatrix},   
        \Fbar(\mcX,N(\mcX)) =
        \begin{bmatrix}
        \Fbar(\BFx^{0},N(\BFx^{0})) \\
        \Fbar(\BFx^{1},N(\BFx^{1})) \\
        \vdots \\
        \Fbar(\BFx^{p},N(\BFx^{p})) \\
        \end{bmatrix}.
    \end{equation*}
    We note $\BFx^{0}:=\BFx$. If the matrix $\sfM(\Phi, \mcX)$ is nonsingular, the set $\mcX$ is poised in $\mcB(\BFx;\Delta)$. 
    Then, the function $M:\mcB(\BFx;\Delta) \to \mbR$, defined as $M(\BFx) = \sum_{j=0}^{p} \beta_j \phi_{j}(\BFx)$ is a stochastic polynomial interpolation model of $f$ on $\mcB(\BFx;\Delta)$. 
    Let $\BFG=
    \begin{bmatrix}
    \beta_1 & \beta_2 & \cdots & \beta_d 
    \end{bmatrix}^\intercal$ be the subvector of $\BFbeta$ and $\sfH$ be a symmetric matrix of size $d\times d$ with elements uniquely defined by $\begin{bmatrix}
    \beta_{d+1} & \beta_{d+2} & \cdots & \beta_p 
    \end{bmatrix}^\intercal$. 
    Then, we can define the stochastic quadratic model $M: \mcB(\BFx;\Delta)\to\mbR$, as
\begin{equation}
    M(\BFx+\BFs) = \beta_0 +  \BFs^\intercal \BFG + \frac{1}{2} \BFs^\intercal \sfH\BFs.\label{eq:mdefn}
\end{equation}
\label{defn:polyintermd}

\end{definition}
\begin{definition} (stochastic quadratic models with diagonal Hessians) A special case of (\ref{eq:mdefn}) is when the Hessian has only diagonal values, i.e.,
\begin{equation}
    \sfH = 
    \begin{bmatrix}
     H_{1} &  & \bold{0}   \\
     & \ddots &   \\
     \bold{0} &  &   H_{d}  \\
    \end{bmatrix}
    \in \mbR^{d\times d}. \label{eq:DH}
\end{equation}
In the stochastic quadratic interpolation model with diagonal Hessian, $p=2d$, the model \eqref{eq:mdefn} contains $2d+1$ unknowns, and $2d+1$ function value estimations are needed to uniquely determine the $\BFG$ and $\sfH$, letting the interpolation set be 
$$\mcX_{cb} = \{\BFx^{0}, \BFx^{0}+\BFe^{1} \Delta, \dots, \BFx^{0}+\BFe^{d} \Delta, \BFx^{0}-\BFe^{1} \Delta, \dots, \BFx^{0}-\BFe^{d} \Delta \}$$ contained in $\mcB(\BFx^{0};\Delta)$. Since the coordinate basis is used to generate the interpolation set, $\BFbeta$ is guaranteed to exist. Hence, $H_i=\alpha_{d+i}\le \infty$ for all $i=1,2,\dots,d$. In this case, $\Phi(\BFx) := (1,x_1, x_2, \dots, x_d, x_1^2, x_2^2, \dots, x_d^2)$,
and $M(\cdot)$ is said to be a stochastic quadratic model with diagonal Hessian. 
\label{defn:diagHess}
\end{definition}

\begin{definition} (stochastic fully linear models) Given $\BFx \in \mbR^d$ and $\Delta > 0$, a function $M:\mcB(\BFx;\Delta) \to \mbR$ obtained following Definition~\ref{defn:diagHess} is the stochastic fully linear model of $f$ on $\mcB(\BFX;\Delta)$ if $\TD f$ is Lipschitz continuous with constant $\kappa_{L}$, and there exist positive constants $\kappa_{eg}$ and $\kappa_{ef}$ dependent on $\kappa_{L}$ but independent of $\BFx$ and $\Delta$ such that almost surely
\begin{equation*}
    \|\TD f(\BFx) - \TD M(\BFx)\| \le \kappa_{eg}\Delta\text{ and } |f(\BFx) - M(\BFx)| \le \kappa_{ef}\Delta^2\ \forall \BFx\in\mcB(\BFx;\Delta).
\end{equation*}
\label{defn:fullylinear}
\end{definition}

\begin{definition} (Cauchy reduction) Given $\BFx \in \mbR^d$ and $\Delta > 0$, a function $M:\mcB(\BFx;\Delta) \to \mbR$ obtained following Definition~\ref{defn:diagHess}, $\BFs^c$ is called the Cauchy step if 
    \begin{equation}
        M(\BFx)-M(\BFx+\BFs^c) \ge \frac{1}{2}\|\TD M(\BFx)\|\left( \frac{\|\TD M(\BFx)\|}{\|\TD^2 M(\BFx)\|}\land \Delta \right).
    \end{equation}
    We assume that $\|\TD M(\BFx)\|/\|\TD^2 M(\BFx)\|=+\infty$ when $\|\TD^2 M(\BFx)\|=0$. The Cauchy step is obtained by minimizing the model $M(\cdot)$ along the steepest descent direction within $\mcB(\BFx;\Delta)$ and hence easy and quick to obtain.
\label{defn:cauchyred}
\end{definition}


\begin{definition} (Filtration and Stopping Time). A filtration $\{\mcF_{k}\}_{k \geq 1}$ over a probability space $(\Omega,\mathbb{P},\mcF)$ is defined as an increasing family of $\sigma$-algebras of $\mcF$, i.e., $\mcF_{k} \subset \mcF_{k+1} \subset \mcF$ for all $k$. We interpret $\mcF_{k}$ as ``all the information available at time $k$." A filtered space $(\Omega,\mathbb{P},\{\mcF_{k}\}_{k\geq1},\mcF)$ is a probability space equipped with a filtration.
A map $N:\Omega \rightarrow \{0,1,2,\dots,\infty\}$ is called a stopping time with respect to $\mcF_{k}$ if the event $\{N = n\} := \{\omega: N(\omega) = n\} \in \mcF_{k}$ for all $n \leq \infty$.
\end{definition}

\section{ASTRO-DF with Coordinate Direct Search}
\label{sec:algorithm}
The ASTRO-DF utilizing a coordinate direct search has four repeating stages: (a) diagonal Hessian local model construction via interpolation on function value estimates of coordinate basis set with adaptive sample sizes, (b) model candidate  identification via an inexact optimization of the local model, (c) function value estimation at the model candidate solution and  direct search candidate identification, (d) trust-region radius and incumbent updating. Algorithm~\ref{alg:ASTRO-DF-refined} details the operations in each iteration.

\begin{algorithm}[htp]  
\caption{ASTRO-DF with Coordinate Direct Search }
\label{alg:ASTRO-DF-refined}
\begin{algorithmic}[1]
\REQUIRE Initial guess $\BFx_{0}\in\mbR^d$, initial and maximum trust-region radius $\Delta_{0},\Delta_{\max}>0$, model ``fitness'' and certification thresholds $\eta\in(0,1)$ and $\mu>0$, sufficient reduction constant $\theta>0$, expansion and shrinkage constants $\gamma_1>1$ and $\gamma_2\in(0,1)$, sample size lower bound $\lambda_k$, and adaptive sampling constant $\kappa>0$.
\FOR{$k=0,1,2,\hdots$}
 
\STATE \label{ASTRO:TRsubprob} \textit{Model Construction:} Select the design set $\mcX_{k}=\left\{ \BFX_{k}^{i}\right\}_{i=0}^{2d}\subset\mcB(\BFX_{k};\Delta_{k})$ following Definition \eqref{eq:mdefn} and estimate $\Fbar\left(\BFX_{k}^{i},N\left(\BFX_{k}^{i}\right)\right)$, where
\begin{equation}
N\left(\BFX_{k}^{i}\right)=\min\biggl\{ n\geq \lambda_k:\frac{\sigmahat\left(\BFX_{k}^{i},n\right)}{\sqrt{n}}\leq\frac{\kappa\Delta_{k}^{2}}{\sqrt{\lambda_k}}\biggr\},\label{eq:inner-sampling-interpolation}
\end{equation} for $i=0,1,\ldots,2d$ ($\BFX_k^0=\BFX_k$) and construct the model $M_{k}\left(\BFX_{k}+\BFs\right)$ via interpolation.

\STATE \label{ASTRO:TRsubprob} \textit{Subproblem:} Approximate the $k$-th step by minimizing the model in the trust-region, $\BFS_{k}=\argmin_{\left\| \BFs\right\| \leq \Delta_{k}}M_{k}(\BFX_k+\BFs)$, and set $\BFXtilde_{k+1}=\BFX_k+\BFS_{k}$. 
\STATE \label{ASTRO:evaluate}\textit{Candidate Evaluation:} Use adaptive sampling rule~\eqref{eq:inner-sampling-interpolation}  to estimate $\Fbar(\BFXtilde_{k+1},\Ntilde_{k+1})$. Define the best design point $\BFXhat_{k+1}=\argmin_{\BFx\in\mathcal{X}_{k}\backslash\BFX_k}\Fbar(\BFx, N(\BFx))$, its sample size $\hat{N}_{k+1} = N(\BFXhat_{k+1})$, incumbent's sample size $\Nhat_k = N\left(\BFX_k\right)$, direct search reduction $\Rhat_k=\Fbar(\BFX_{k},\Nhat_{k})-\Fbar (\BFXhat_{k+1},\Nhat_{k+1})$, subproblem reduction   $\Rtilde_k=\Fbar(\BFX_{k},\Nhat_{k})-\Fbar (\BFXtilde_{k+1},\Ntilde_{k+1})$, and model reduction $R_k=M_{k}(\BFX_{k})-M_{k}(\BFXtilde_{k+1})$. 
\STATE \label{ASTRO:ratio} \textit{Update:} Set
$(\BFX_{k+1},N_{k+1},\Delta_{k+1})=$ \[
\begin{cases}
    (\BFXhat_{k+1},\Nhat_{k+1},\gamma_1\Delta_{k}\land \Delta_{\max})& \text{if } \Rhat_{k}>\max\{\Rtilde_k,\theta\Delta_{k}^2\}, \\
    (\BFXtilde_{k+1},\Ntilde_{k+1},\gamma_1\Delta_{k}\land \Delta_{\max})              & \text{else if }\Rtilde_{k}\geq\eta R_{k} \text{ and }\mu\|\TD M_{k}(\BFX_{k})\|\ge\Delta_{k},\\
    (\BFX_{k},\Nhat_{k},\Delta_{k}\gamma_2)              & \text{otherwise},
\end{cases}
\] and $k=k+1$.
\ENDFOR
\end{algorithmic}
\label{alg:main-stoch}
\end{algorithm}

The main changes from the original ASTRO-DF algorithm, as developed by \cite{Sara2018ASTRO}, and their consequent advantages are listed below:
\paragraph{Model Construction} The point selection is deterministic and the local model $M_k(\cdot)$ is constructed with $2d$ new points along the coordinate basis at the trust-region boundary in each iteration to obtain a diagonal Hessian model. As mentioned in Subsection \ref{sec:contribution}, the suggested design set has several advantages. Firstly, the model gradient error for all points within a trust-region has a similar upper bound but with fewer design points. We will also show in Theorem \ref{thm:modelerror_stoch_iterate} that the gradient error at the center point improves to $\mcO(\Delta_k^2+E/\Delta_k)$ from $\mcO(\Delta_k+E/\Delta_k)$. When $\Delta_k<1$, the stochastic errors are allowed to be larger by $\mcO(\Delta_k^2-\Delta)$ to have the same upper bound, resulting in a smaller sample size. Therefore, by utilizing a fixed geometry following a coordinate basis, we achieve the same upper bound for the gradient error while using fewer design points and smaller sample sizes for each point. In other words, the model quality increases with less work! 
Secondly, the design set via coordinate basis is optimal in the sense that it minimizes the constant of well-poisedness in a ball, i.e., well-spread within the trust-region, saving the linear algebra costs to make a well-poised set \citep{tom2023optimalpoised}. The computational complexity of the algorithm for ensuring well-poisedness also depends on the problem dimension. For example, Algorithms 6.1-6.3 in \cite{katya:DFObook} need $\mcO(d^2)$ operations for a linear model and $\mcO(d^6)$ for a quadratic model. Lastly, when using the predetermined design set, the Vandermonde matrix $\sfM(\Phi, \mcY)$ exhibits a unique structure, resulting in a reduced matrix inversion cost of $\mcO(d^2)$ instead of $\mcO(d^3)$.

Furthermore, the certification loop in the model construction of the original ASTRO-DF, which repeatedly shrinks $\Delta_k$ and rebuilds the model with new points in the smaller neighborhood, has now been removed. Instead, the criteria are checked once at the points to determine whether to accept the model candidate if the interpolation candidate fails to pass its sufficient reduction test.

\paragraph{Updating next incumbent} The next incumbent is determined upon estimating the function at the subproblem's minimizer. If the best among $2d+2$ points (interpolation points, incumbent, and model candidate) is an interpolation point that satisfies a reduction of at least $\theta\Delta_k^2$ in the estimated function value, that point becomes the next incumbent; otherwise, the model candidate is accepted if the decrease in the estimated function value is at least $\eta$ times the reduction in the model and the model passes its certification criteria of $\mu\|\TD M_k(\BFX_k)\|\geq\Delta_k$. 


Skipping model construction and choosing the incumbent from the design points \emph{directly} turns this algorithm into a mesh adaptive direct search akin to NOMAD~\citep{audet2021nomad}. Direct search is one of the numerical optimization algorithms that, unlike trust regions or line search that use derivative information, only relies on function evaluations of a design set to determine whether to accept or reject a design point as the next incumbent. 
In ASTRO-DF, choosing a design point that outperforms the model candidate point, albeit with a threshold of sufficient reduction, as the next incumbent will increase \textit{the probability of success}. The threshold of sufficient reduction is necessary because a model candidate point is considered qualified if the model is fit, i.e., provides good prediction of the function that renders the improvement in the function estimate a factor of $\Delta_k^2$, provided that the model is certified. In the absence of a model, one needs an alternative criteria to maintain the quality of the next incumbent. But interestingly, the inclusion of direct search also enables accepting a point that would \emph{not} yield a reduction that is more than $\eta R_k$ if that quantity is more than $\theta\Delta_k^2$. In other words, acceptance criteria for the new incumbent becomes less strict. However, note that such improved success probability takes effect in the early iterations, hence helping the finite-time performance. As the iterations progress, the model tends to approximate the function more accurately in smaller neighborhoods, reducing the likelihood of the model candidate point underperforming the interpolation points.

\paragraph{Simulation budget} The deterministic lower bound sequence for the sample sizes, i.e., $\{\lambda_k,k\in\mbN\}$ now grows logarithmically instead of linearly with $k$. This leads to a slower growth in the minimum sample size and savings of the simulation budget.

\section{Consistency and Complexity Analysis} \label{sec:complexity} 
This section presents the convergence and complexity analysis of ASTRO-DF with coordinate direct search. We first list the assumptions and useful results from the original ASTRO-DF. The needed assumptions to obtain the convergence and complexity results are standard and not restrictive, allowing for use of Bernstein inequality to bound the tail of sums of subexponential random variables that are not independent of one another. 

\begin{assumption} (function)
    The function $f$ is twice continuously differentiable in an open domain $\Omega$, $\TD f$ is Lipschitz continuous in $\Omega$ with constant $\kappa_{Lg}>0$, and $\TD^2 f$ is Lipschitz continuous in $\Omega$ with constant $\kappa_{L}>0$.
    \label{assum:fn}
\end{assumption}

\begin{assumption} (stochastic noise)
    The Monte Carlo oracle generates iid random variables $F(\BFX^{i}_k,\xi_j) = f(\BFX^{i}_k)+E^{i}_j$ with $E^{i}_j\in\mcF_{k,j}$,where $\mcF_k:=\mcF_{k,0}\subset\mcF_{k,1}\subset\dots\subset \mcF_{k+1}$ for all $k$. In addition, let the design set be $\{\BFX_k^{i}\}_{i=0,1,\cdots,p}\in\mcF_{k-1}$. 
    Then the stochastic errors $E_{k,j}^{i}$ are independent of $\mcF_{k-1}$, $\mbE[E_{k,j}^{i} \left\vert \mcF_{k,j-1} \right.]=0$,
    and there exists $\sigma^2>0$ and $b>0$ such that for a fixed $n$, $$\frac{1}{n} \sum_{j=1}^n \mbE[|E_{k,j}^{i}|^m \left\vert \mcF_{k,j-1}\right.] \leq\frac{m!}{2}b^{m-2}\sigma^2,\ \forall m=2,3,\cdots,\forall k.$$
    Furthermore, for a given solution at iteration $k$ and constant $c>0$, there exists a large  $c_0>0$ such that for all $\lambda_k\leq n\leq N(\BFX_k^i)$,  $$\limsup_{c\to\infty}\sup_{c_0\leq t\leq c}\frac{\mbP\{\frac{1}{n-1}\sum_{j=1}^{n-1}|E_{k,j}^i|>c-t \mid |E_{k,n}^i|=t\}}{\mbP\{\frac{1}{n-1}\sum_{j=1}^{n-1}|E_{k,j}^i|>c-t\}}<\infty.$$ In other words, the ratio above is $\mcO(1)$ for all $t\in[c_0,c]$.
    \label{assum:error}
\end{assumption}
This assumption enables characterization of the tail probability of a sequence of dependent random variables that each have a sub-exponential like tail behavior in order for the Bernstein inequality bounds to be applicable. The criteria for the stochastic noise in Assumption~\ref{assum:error} are less stringent than boundedness or sub-Gaussian noise. See~\cite{ha2023siamopt} for more details.

\begin{assumption} (model)
    There exists some constant $\kappa_{fcd}\in(0,1]$ such that for all $k$, $M_k(\BFX_k) - M_k(\BFX_k+\BFS_k) \ge \kappa_{fcd}[M_k(\BFX_k) - M_k(\BFX_k+\BFS^c_k)]$, where $\BFS_k^c$ is the Cauchy step. Additionally, there exists $\kappa_\sfH\in(0,\infty)$ such that $\|\sfH_k\|\leq\kappa_\sfH$ for all $k$ with probability 1.\label{assum:fcd}
\end{assumption}

\begin{assumption} (sample size)
    The ``lower-bound sequence" $\{\lambda_k\}$ on the adaptive sample sizes satisfies $(\log k)^{1+\epsilon_\lambda}= \mcO(\lambda_k)$ 
    for some $\epsilon_\lambda \in (0,1)$.
    \label{assum:lambda}
\end{assumption}

\subsection{Useful Existing Results}
In this section, we introduce several useful results to obtain the almost sure convergence and complexity of ASTRO-DF with coordinate direct search. The first result implies that the estimation error with adaptive sampling is sufficiently small for large enough iterations with probability 1. 

\begin{theorem} \label{thm:bddE} 
    (Theorem 2.7 by \cite{ha2023siamopt})
    Let Assumptions \ref{assum:error} and \ref{assum:lambda} hold and define  $|\Ebar_k|:=N_k^{-1}\sum_{j=1}^{N_k}E_j$. Then for a given $c>0$, we have that $|\Ebar_k|\leq c \Delta_k^2$ for sufficiently large $k$ almost surely. In other words, $\mbP\{|\Ebar_k|\geq c\Delta_k^2\text{ i.o.}\}=0.$
\end{theorem}

\begin{remark}
    Because of the way $N_k$ is determined and Assumption \ref{assum:error} and \ref{assum:lambda}, one can use the Bernstein inequality to write $\mbP\{|\Ebar_k|\geq c\Delta_k^2\ \mid\mcF_k\}\leq \alpha k^{-1+\epsilon_\lambda}$ for any $c>0$ and some $\alpha>0$; i.e., the estimate is sufficiently accurate with a probability that tends to one as $k\to\infty$. In contrast to the assumption of probababilisitcally accurate estimates with a fixed probability and fixed accuracy threshold~\citep{chen2018storm}, the estimators' property here is ensured as a result of adaptive sampling rule devised and not assumed.
\end{remark}

The second existing result characterizes the stochastic model error with estimation error when the stochastic model is constructed by Definition \ref{defn:diagHess}.

\begin{lemma} \label{lemma:stochastic_interp}
    (Lemma 2.9 by \cite{Sara2018ASTRO} and Proposition 3.1 by \cite{tom2023optimalpoised})
    Let Assumption \ref{assum:fn} hold and let $M(\cdot)$ be a stochastic quadratic interpolation model with diagonal Hessian of $f$ on $\mcB(\BFX;\Delta)$. Let $m(\cdot)$ be the corresponding deterministic polynomial interpolation model of $f$ on $\mcB(\BFX;\Delta)$. Then for all $\BFx \in \mcB(\BFX;\Delta)$,
    \begin{equation*}
        |M(\BFx)-m(\BFx)|\le (2d+1)\max_{i=0,1,\dots,2d}|\Fbar(\BFX^{i},n(\BFX^{i}))-f(\BFX^{i})|.
    \end{equation*}
\end{lemma}

\begin{remark}
    A direct observation from Lemma~\ref{lemma:stochastic_interp}, Theorem~\ref{thm:modelerror_stoch_iterate}, and Appendix \ref{app:kappa-ef} is that $$\mbP\{\|\TD f(\BFx)-\TD M_k(\BFx)\|\geq(c_{eg}+\kappa_{eg})\Delta_k\text{ and }| f(\BFx)-M_k(\BFx)|\geq(c_{ef}+\kappa_{ef})\Delta_k^2\mid\mcF_k\}\leq\alpha' k^{-1+\epsilon_\lambda},$$ for some $\alpha'>0$, $c_{eg}>0$, and $c_{ef}>0$ and for all $\BFx\in\mcB(\BFX_k;\Delta_k)$. In other words, the stochastic model is probabilistically fully linear with a probability that goes to one as $k\to\infty$. Again, despite the resemblance to the assumption of probabilistically fully linear models with fixed probability~\citep{chen2018storm}, this result is ensured without overburdening the computation, as we will see in the obtained complexity results. 
\end{remark}

\subsection{Model Quality}
To prove consistency, we first look at the model quality with the suggested design set in Definition \ref{defn:diagHess}. In general, the stochastic polynomial interpolation model has achieved that the model gradient norm is $\mcO(\Delta_k+E/\Delta_k)$ accurate for any point within the trust region of size $\Delta_k$ \citep{Sara2018ASTRO}. The following Theorem gives a bound on the stochastic model error that holds for any point within the trust region. 
    
\begin{theorem} \label{thm:modelerror_stoch}
    Let Assumption \ref{assum:fn} and \ref{assum:error} hold, and $M_k(\BFx)$ be a stochastic quadratic model of $f$ on $\mcB(\BFX_k;\Delta_k)$ with diagonal Hessian built on $\Fbar\left(\BFX_k^{i},n\left(\BFX_k^{i}\right)\right)=f\left(\BFX_k^{i}\right) + \Ebar_k^{i},$ for $i=0,1,\dots,2d$ on iteration $k$. With $\TD M_k(\BFx) = \left(\BFx-\BFX_k\right)^\intercal \sfH_k + \BFG_k$ and $\kappa_{eg1}$ = $\frac{5\sqrt{2d}}{2} (\kappa_{Lg} + \kappa_{\sfH})$, we can uniformly bound the model gradient error for all points $\BFx$ in $\mcB(\BFX_k;\Delta_k)$ by
    $$\| \TD M_k(\BFx) - \TD f(\BFx) \| \le \kappa_{eg1} \Delta_k + \frac{\sqrt{\sum_{i=1}^{2d}\left(\Ebar_k^{i}-\Ebar_k^{0}\right)^2}}{\Delta_k}.$$
\end{theorem}
\begin{proof}
    With a stochastic quadratic models with diagonal Hessian, the theorem holds following the same steps of Appendix B in \cite{Sara2018ASTRO}. 
\end{proof}

We next investigate the model gradient error at the current iterate as a special case. At the center point, we can get a tighter upper bound like a central finite difference. Theorem \ref{thm:modelerror_stoch_iterate} shows the model gradient error at an  iterate $k$ is $\mcO(\Delta_k^2 + E/\Delta_k)$.
    
\begin{theorem} \label{thm:modelerror_stoch_iterate}
    Let Assumption \ref{assum:fn} and \ref{assum:error} hold, and the interpolation model $M_k(\BFx)$ of $f$ be a stochastic quadratic models with diagonal Hessian constructed using $\Fbar\left(\BFX_k^{i},n\left(\BFX_k^{i}\right)\right)=f\left(\BFX_k^{i}\right) + \Ebar_k^{i},$ for $i=0,1,\dots,2d$. Then, with $\kappa_{eg2}$ = $\frac{\sqrt{d}}{6} \kappa_{L}$, we can uniformly bound the model gradient error at the center point by
    $$\| \BFG_k - \TD f(\BFX_k) \| \le \kappa_{eg2} \Delta_k^2 + \frac{\sqrt{\sum_{i=1}^{d}\left(\Ebar_k^{i}-\Ebar_k^{i+d}\right)^2}}{2\Delta_k}.$$
\end{theorem}
\begin{proof}
    We begin by decomposing $\BFG_k = \TD f(\BFX_k) + E^g(\BFX_k) + e^g(\BFX_k),$ 
    where $\BFE^g_k := \BFG_k - \BFg_k$ and $\BFe^g_k := \BFg_k - \TD f(\BFX_k)$ are the stochastic model error and deterministic model error, respectively and $\BFg_k$ is the deterministic model gradient. 
    Then we can obtain $\BFG_k$ by finding $\BFbeta$ such that $\widetilde{\sfM}(\Phi, \mcX_{k}) \BFbeta = \widetilde{\BFF}$, where
    \begin{equation*}
        \widetilde{\sfM}(\Phi, \mcX_{k}) = 
        \begin{bmatrix}
        1 & 0 & 0 & \cdots & 0 & 0 & 0 & \cdots & 0  \\
        1 & \Delta_k & 0 & \cdots & 0 & \Delta_k^2/2 & 0 & \cdots & 0  \\
        1 & 0 & \Delta_k & \cdots & 0 & 0 & \Delta_k^2/2 & \cdots & 0  \\
        \vdots & \vdots & \vdots & \vdots & \vdots  & \vdots& \vdots & \vdots & \vdots\\
        1 & 0 & 0 & \cdots & \Delta_k & 0 & 0 & \cdots & \Delta_k^2/2  \\
        1 & -\Delta_k & 0 & \cdots & 0 & \Delta_k^2/2 & 0 & \cdots & 0  \\
        1 & 0 & -\Delta_k & \cdots & 0 & 0 & \Delta_k^2/2 & \cdots & 0  \\
        \vdots & \vdots & \vdots & \vdots & \vdots  & \vdots& \vdots & \vdots & \vdots\\
        1 & 0 & 0 & \cdots & -\Delta_k & 0 & 0 & \cdots & \Delta_k^2/2  \\
        \end{bmatrix},
    \end{equation*} and
    \begin{equation*}
        \widetilde{\BFF} =
        \begin{bmatrix}
        \Fbar(\BFX_k^{0},N(\BFX_k^{0})) \\
        \Fbar(\BFX_k^{1},N(\BFX_k^{1})) \\
        \vdots \\
        \Fbar(\BFX_k^{2d},N(\BFX_k^{2d}))\\
        \end{bmatrix}.
    \end{equation*}
    As a result, the $i$-th element of the gradient estimate is obtained by
    \begin{equation}
    \label{eq:G-k-i}
        [\BFG_k]_i = \frac{f(\BFX_k+\BFe_i\Delta_k) - f(\BFX_k-\BFe_i\Delta_k)}{2\Delta_k} + \frac{\Ebar_k^{i}-\Ebar_k^{i+d}}{2\Delta_k}.
    \end{equation}
    Based on Taylor expansion, it is well known that
    \begin{equation*} \label{eq:deterministic_gradient_error}
        \|\BFg_k - \TD f(\BFX_k)\| \le \frac{\sqrt{d}}{6}\kappa_{L}\Delta_k^2,
    \end{equation*}
    with \eqref{eq:G-k-i} like the central finite difference \citep{berahas2021gradientest, coope2020gradient}. Hence, we obtain $[\BFE^g_k]_i = (\Ebar_k^{i}-\Ebar_k^{i+d})/{2\Delta_k}$ for any $i\in\{1,\dots,d\}$, which implies that $\|\BFE^g_k\| = \sqrt{\sum_{i=1}^d(\Ebar_k^{i}-\Ebar_k^{i+d})^2}/(2\Delta_k)$. As a result, 
    \begin{align*}
        \|\BFG_k - \TD f(\BFX_k)\| &\le \|\BFg_k - \TD f(\BFX_k)\| + \|\BFG_k - \BFg_k\| \\
        &\le \frac{\sqrt{d}}{6}\kappa_{L}\Delta_k^2 + \frac{ \sqrt{\sum_{i=1}^d(\Ebar_k^{i}-\Ebar_k^{i+d})^2}}{2\Delta_k}.
    \end{align*}
\end{proof}

\subsection{Convergence}
For the remainder of the paper, we define the set $\mcK=\{k\in\mbN:\text{iteration}\;k\; \text{is  successful}\}$. 

\begin{theorem} \label{thm:deltaconverge}
    Let Assumptions \ref{assum:fn}-\ref{assum:lambda} hold. Then, $\Delta_k \xrightarrow[]{w.p.1} 0 \text{ as } k \rightarrow \infty.$
\end{theorem}
\begin{proof}
    If $\mcK$ is finite, the trust-region radius tends to zero due to infinite shrinkage within unsuccessful iterations, making the statement of the theorem hold trivially. So we consider that $\mcK$ is infinite. The candidate solution can be one of the design point among the design set or the solution from the trust-region subproblem. Let us define two different set $\mcK_{ds} := \mcK \cap \{k\in\mbN:\BFX_{k+1} = \BFXhat_{k+1}\}$ and $\mcK_{tr} := \mcK \cap \{k\in\mbN:\BFX_{k+1} = \BFXtilde_{k+1}\}$. Then we have, for any $k \in \mcK_{ds}$, 
    \begin{equation*}
        \Rhat_k=\Fbar(\BFX_{k},N_k) - \Fbar(\hat\BFX_{k+1},\hat{N}_{k+1}) 
        = M_k(\BFX_k) - M_k(\hat\BFX_{k+1})
        \ge \theta\Delta_k^2,
    \end{equation*}
    and for any $k\in\mcK_{tr}$,
    \begin{align*}
        \Rtilde_k=\Fbar(\BFX_{k},N_k) - \Fbar(\tilde\BFX_{k+1},\tilde{N}_{k+1}) & \ge \eta(M_k(\BFX_k) - M_k(\tilde\BFX_{k+1}))\\
        &\ge \frac{\eta\kappa_{fcd}}{2}\frac{\Delta_k}{\mu}\left(\frac{\Delta_k}{\mu\kappa_\sfH}\land{\Delta}_k\right)\\
        &\ge \left(\frac{\eta\kappa_{fcd}}{2\mu}\left(\frac{1}{\mu\kappa_\sfH}\land 1\right)\right)\Delta_k^2.
    \end{align*} Then, for any $k\in\mcK$,
    \begin{align*}
    \theta'\sum_{\substack{k\in\mcK }}\Delta_k^2 &\le \sum_{\substack{k\in\mcK }} (f(\BFX_{k}) - f(\BFX_{k+1}) + \Ebar_{k}-\Ebar_{k+1})\le f(\BFx_0) - f^* +\sum_{k=0}^{\infty}(|\Ebar_{k}|+|\Ebar_{k+1}|),
    \end{align*}
    where $\theta'=\left(\frac{\eta\kappa_{fcd}}{2\mu}(\frac{1}{\mu\kappa_\sfH}\land 1)\land\theta\right)$. Let $\mcK = \{k_1,k_2,\dots\}$, $k_0 = -1,$ and $\Delta_{-1}=\Delta_0/\gamma_2$. Then from the fact that $\Delta_k\le \gamma_1\gamma_2^{k-k_i-1}\Delta_{k_i}$ for $k=k_i+1,\dots,k_{i+1}$ and each $i$, we obtain 
    \begin{equation*}   
        \sum_{k=k_i+1}^{k_{i+1}}\Delta_k^2 \le \gamma_1^2 \Delta_{k_i}^2 \sum_{k=k_i+1}^{k_{i+1}} \gamma_2^{2(k-k_i-1)} \le \gamma_1^2\Delta_{k_i}^2\sum_{k=0}^{\infty}\gamma_2^{2k} = \frac{\gamma_1^2}{1-\gamma_2^2}\Delta_{k_i}^2.
    \end{equation*}
    Then, we have 
    \begin{align*}
    \sum_{k=0}^{\infty}\Delta_k^2 \le \frac{\gamma_1^2}{1-\gamma_2^2}\sum_{i=0}^{\infty}\Delta_{k_i}^2< \frac{\gamma_1^2}{1-\gamma_2^2}\left(\frac{\Delta_0^2}{\gamma_2^2}+\frac{f(\BFx_0)-f^*+E'_{0,\infty}}{\theta'}\right),
    \end{align*}
    where $E'_{i,j} = \sum_{k=i}^{j}(|\Ebar_{k}|+|\Ebar_{k+1}|)$. By Theorem \ref{thm:bddE}, there must exists a sufficiently large $K_\Delta$ such that $|\Ebar_{k}|+|\Ebar_{k+1}| < c_{\Delta}\Delta_k^2$ for any given $c_{\Delta} > 0$ and every $k \ge K_\Delta$. Then we obtain 
    \begin{equation*}
    \sum_{k=0}^{K_\Delta-1}\Delta_k^2 + \sum_{k=K_\Delta}^{\infty}\Delta_k^2 < \frac{\gamma_1^2}{1-\gamma_2^2}\left(\frac{\Delta_0^2}{\gamma_2^2}+\frac{f(\BFx_0)-f^*+E'_{0,K_\Delta-1}+E'_{K_\Delta,\infty}}{\theta'}\right),
    \end{equation*}
    and $E'_{K_\Delta,\infty}=\sum_{k=K_\Delta}^\infty|\Ebar_{k}|+|\Ebar_{k+1}| < \sum_{k=K_\Delta}^\infty c_{\Delta}\Delta_k^2.$ As a result, we get 
    \begin{align} \sum_{k=K_\Delta}^{\infty}\Delta_k^2 < \frac{\gamma_1^2}{1-\gamma_2^2}\left(\frac{\Delta_0^2}{\gamma_2^2}+\frac{f(\BFx_0)-f^*+E'_{0,K_\Delta-1}}{\theta'}\right)\left(1-\frac{\gamma_1^2}{1-\gamma_2^2}\frac{c_\Delta}{\theta'}\right)^{-1}<\infty.\label{eq:delta-upper-bound}
    \end{align}
   Therefore, $\Delta_k \xrightarrow[]{wp1} 0\text{ as } k \rightarrow \infty.$
\end{proof}

Now we prove the almost sure convergence of ASTRO-DF with coordinate direct search.
\begin{theorem} \label{thm:almostsureconvergence}
    Let Assumptions \ref{assum:fn}-\ref{assum:lambda} hold. Then, $\|\TD f(\BFX_k) \| \xrightarrow[]{w.p.1} 0 \text{ as } k \rightarrow \infty.$
\end{theorem}
\begin{proof}

    Let us define the set 
    \begin{equation*}
        \mcV := \left\{ \exists \,\text{inf}. \, \text{subseq}.\, \{k_j\} \, \text{ s.t. } \, \left( \frac{\Delta_{k_j}}{\|\BFG_{k_j}\|}\le \min\left\{\frac{\eta'}{2\kappa_{ef}},\mu\right\}\right) \bigcap \left(\rhohat_k<\eta\right)\right\},
    \end{equation*}
     where $\eta'= 6^{-1}(1-\eta)\kappa_{fcd}((\mu\kappa_{\sfH})^{-1}\land 1)$. Then, we have $\mbP\{\mcV\}=0.$ The proof trivially follows from  Lemma 5.2 in \cite{Sara2018ASTRO} by considering that $\mu \|\BFG_{k_j}\| \ge \Delta_{k_j}$ is now in the set $\mcV$, which was ensured by the criticality step in \cite{Sara2018ASTRO}. We also have $\| \BFG_k - \TD f(\BFX_k) \| \xrightarrow[]{wp1}0 \text{ as } k\rightarrow\infty.$ The proof follows from that of Lemma 5.4 by \cite{Sara2018ASTRO}, considering that we always use $\Delta_k$ for iteration $k$. Now we will prove that Algorithm \ref{alg:main-stoch} obtains $\liminf \|\TD f(\BFX_k) \| \xrightarrow[]{wp1} 0 \text{ as } k \rightarrow \infty$ with $\| \BFG_k - \TD f(\BFX_k) \| \xrightarrow[]{wp1}0 \text{ as } k\rightarrow\infty$ and Theorem \ref{thm:deltaconverge}.
     For the purpose of arriving at a contradiction, suppose the set $\mcD_g=\{\omega: \exists\, \kappa_{lbg}(\omega),k_g(\omega)>0 \text{ s.t. } \| \BFG_k\| \ge \kappa_{lbg}(\omega) \, \forall k > k_g(\omega) \}$,
    has positive measure. Due to the assumptions of the theorem, we can find a set $\mcD_d$ of sample paths such that $\mbP\{\mcD_d\}=1$, and such that for each $\omega \in \mcD_d$, $\Delta_k(\omega)\rightarrow 0$ and $\omega \in \mcV^c$.
    Let $\omega \in \mcD_g \cap \mcD_d$. Then either $\| \BFG_k(\omega)\| < \min\{\eta'(2\kappa_{ef})^{-1},\mu\}^{-1}\Delta_k(\omega)$ or $\rhohat_k(\omega)>\eta$ for large enough $k$. Since $\Delta_k$ goes to zero almost surely, $\| \BFG_k(\omega)\| < \min\{\eta'(2\kappa_{ef})^{-1},\mu\}^{-1}\Delta_k(\omega)$ cannot be true for large enough $k$. Therefore, for any $\omega \in \mcD_g \cap \mcD_d$, it must be true that $\rhohat_k(\omega) \ge \eta$ for large enough $k$. In other words, the iterations in sample path $\omega \in \mcD_g \cap \mcD_d$ are eventually successful. 

    Now let $K_s(\omega)>0$ be such that $K_s(\omega)-1$ is the last unsuccessful iteration in sample-path $\omega \in \mcD_g \cap \mcD_d$, that is, $k$ is a successful iteration if $k\ge K_s(\omega)$. Next, $\Delta_k \ge \Delta_{\max\{K_g(\omega), K_s(\omega)\}}$ contradicts the observation $\Delta_k(\omega) \rightarrow 0$. We conclude that $\mbP\{\mcD_g\}=0$ and that $\liminf_{k\rightarrow\infty} \| \BFG_k \| = 0$ almost surely. This along with the fact that $\| \BFG_k - \TD f(\BFX_k) \| \xrightarrow[]{wp1}0 \text{ as } k\rightarrow\infty$, implies $\liminf_{k\rightarrow\infty} \| \TD f(\BFX_k) \| = 0$ almost surely. Then, the almost sure convergence of Algorithm \ref{alg:ASTRO-DF-refined} follows from $\liminf_{k\rightarrow\infty} \| \BFG_k \| = 0$ almost surely, and Theorem \ref{thm:deltaconverge}. The proof is completed by trivially following steps in Theorem 5.5 in \cite{Sara2018ASTRO} and considering that $\Rhat_k
    \ge \theta\Delta_k^2$.
\end{proof}

Now we prove that ASTRO-DF with coordinate direct search has a higher probability of success compared to the original ASTRO-DF. Although this result is not utilized in our current complexity results, it demonstrates that the progress achieved by integrating a coordinate direct search is at least as good as the progress made without that feature, thereby supporting its positive impact on the finite-time performance of the algorithm.
\begin{theorem}
\label{thm:increaseing-prob}
    Let Assumptions \ref{assum:fn}-\ref{assum:lambda} hold. Then given $\BFX_k \in \mbR^d$, the probability of having a successful iteration $k$ for Algorithm \ref{alg:ASTRO-DF-refined} is greater than equal to the probability of having a successful iteration $k$ for the original ASTRO-DF.
\end{theorem}
\begin{proof}
    Let us define two events 
    \begin{align*}
        \mcR:=\{\omega\in\Omega:k(\omega) \text{ is successful with $\BFXhat_{k+1}(\omega)$}|\BFX_k(\omega)=\BFx_k\},\\\mcO:=\{\omega\in\Omega:k(\omega) \text{ is successful with $\BFXtilde_{k+1}(\omega)$}|\BFX_k(\omega)=\BFx_k\}.
    \end{align*} Then, we have $p_o = \mbP\{\mcO\} = \mbP\{\mcO|\mcR\}\mbP\{\mcR\}+\mbP\{\mcO|\mcR^\complement\}\mbP\{\mcR^\complement\},$ where $p_o$ is the probability of having successful iteration $k$ for the ASTRO-DF. In contrast, Algorithm \ref{alg:ASTRO-DF-refined} sequentially check $\BFXhat_{k+1}(\omega)$ and $\BFXtilde_{k+1}(\omega)$. If the iteration $k$ is unsuccessful with $\BFXhat_{k+1}(\omega)$, we will check $\BFXtilde_{k+1}(\omega)$ again. Hence, we have $p_r = \mbP\{\mcR\} + \mbP\{\mcO|\mcR^\complement\},$ where $p_r$ is the probability of having successful iteration $k$ for ASTRO-DF with coordinate direct search. As a result, we have $p_r \ge p_o$.
\end{proof}

\begin{remark}
    Since the sequence $\{\BFX_k\}$ is dependent on previous steps, we cannot directly compare ASTRO-DF with its refined version using Theorem \ref{thm:increaseing-prob}. Nonetheless, Theorem \ref{thm:increaseing-prob} implies that, given an incumbent $\BFX_k$, the refined ASTRO-DF has a higher likelihood of achieving success, preventing $\Delta_k$ from becoming too small too quickly. This enables the algorithm to save significant budget due to $N_k = O(\Delta_k^{-4})$.
\end{remark}

\subsection{Complexity}
While the iteration complexity and work complexity of the original ASTRO-DF has been extensively studied in \cite{ha2023siamopt}, our focus in this section is on the refinements and their impact on the complexity analysis. Specifically, we will examine how these refinements affect the algorithm's computational efficiency.

The following Lemma proves that with too small $\Delta_k$ the iteration $k$ becomes successful almost surely for sufficiently large $k < T_\epsilon$. In other words, the trust region size has a lower bound before $T_\epsilon$, which we can use later to discuss the complexity.

\begin{lemma}\label{thm:delta_epsilon}
    Let Assumptions \ref{assum:fn}-\ref{assum:lambda} hold and $\epsilon>0$ be given. Then there exists $c_{lb}>0$ where   
    \begin{equation*}
        \mbP\left\{\Delta_k < c_{lb}\epsilon \text{ for large }k<T_\epsilon \Rightarrow k \in \mcK \right\} =1.
    \end{equation*}
\end{lemma}
\begin{proof}
    Let $\omega$ be fixed. Define 
    \begin{equation}
    c_E=\frac{1}{2d+2}\left(\frac{\kappa_{fcd}(1-\eta)}{2\mu}\left(\frac{1}{\mu\kappa_\sfH}\land 1\right)-\kappa_{ef}\right).\label{eq:c-E}
    \end{equation} We can find $K_E(\omega)>0$ such that $k\geq K_E(\omega)$ implies $|\Ebar_k(\omega)|\leq c_E\Delta_k^2(\omega)$, which also holds for all solutions visited during iteration $k$ by Theorem \ref{thm:bddE}. Then Theorem \ref{thm:modelerror_stoch_iterate} states for $k\geq K_E(\omega)$,    
    \begin{align*}
        \|\BFG_k(\omega)-\TD f(\BFX_k(\omega))\| & \leq \frac{\sqrt{d}}{6}\kappa_L\Delta_k^2(\omega)+\frac{\sqrt{d}}{\sqrt{2}}c_E\Delta_k(\omega).
    \end{align*} Let $c_{lb}$ be such that $$\frac{1}{c_{lb}}>\frac{1}{\mu}+\frac{\sqrt{d}}{6}\kappa_L\Delta_{\max}+\frac{\sqrt{d}}{\sqrt{2}}c_E.$$ Then, if $\Delta_k(\omega)<c_{lb}\epsilon$ for $k<T_\epsilon(\omega)$, we get 
    \begin{align*}
        \|\BFG_{k}(\omega)\| & \geq \|\TD f(\BFX_{k}(\omega))\|-\|\BFG_{k} (\omega)-\TD f(\BFX_{k}(\omega))\|\\
        & > \left(\frac{1}{c_{lb}}-\frac{\sqrt{d}}{6}\kappa_L\Delta_{\max}+\frac{\sqrt{d}}{\sqrt{2}}c_E\right) \Delta_k(\omega)> \frac{1}{\mu}\Delta_k(\omega),
    \end{align*}
    where we have used $\|\TD f(\BFX_k(\omega))\|>\epsilon$ since $k<T_\epsilon(\omega)$.

    This result confirms that the model quality is eventually good whenever the trust-region becomes too small. To complete the proof we need to show the model will lead to success; a sufficient condition for that is $\Rtilde_k(\omega)>\eta R_k(\omega)$. For ease of exposition we drop $\omega$ in the final step of the proof:
    \begin{align}
        \left|1-\frac{\Rtilde_k}{R_k}\right|&=\left|\frac{\Fbar(\BFXtilde_{k+1},\Ntilde_{k+1})-M_k(\BFXtilde_{k+1})}{M_k(\BFX_k)-M_k(\BFXtilde_{k+1})}\right|\nonumber\\
        &\leq\frac{|\Ebar(\BFXtilde_{k+1})|+|f(\BFXtilde_{k+1})-m_k(\BFXtilde_{k+1})|+|m_k(\BFXtilde_{k+1})-M_k(\BFXtilde_{k+1})|}{\frac{\kappa_{fcd}}{2\mu}(\frac{1}{\mu\kappa_\mcH}\land 1)\Delta_k^2}\nonumber\\
        &\leq\frac{c_E+\kappa_{ef}+(2d+1) c_E}{\frac{\kappa_{fcd}}{2\mu}(\frac{1}{\mu\kappa_\mcH}\land 1)}=1-\eta.\label{eq:success-ratio}
    \end{align}
\end{proof}

\begin{remark} \label{remark:delta_epsilon}
In the proof of Lemma~\ref{thm:delta_epsilon}, the trust region lower bound constant $c_{lb}$ depends on the coefficient of the tolerable estimation error $c_E$, specified in~\eqref{eq:c-E}. On the other hand, the statement of the Lemma does not convey that $\Delta_k<c_{lb}\epsilon$ must occur before $T_\epsilon$. But if it does, it leads to the expansion of the trust region whenever the trust region becomes small enough. So while $\epsilon$ can be chosen to be large such that $\Delta_k$ will not become too small to reach it, it does not cause any problem in the future complexity results. The refined ASTRO-DF leads to a larger $c_E$ since in~\eqref{eq:success-ratio} and \eqref{eq:c-E}, we now have $2d$ instead of $(d+1)(d+2)/2-1$ points. This means that compared to the original ASTRO-DF where $c_E=\mcO(d^{-2})$ and $c_{lb}=\mcO(d)$, this refined version will have $c_E=\mcO(d^{-1})$ and  $c_{lb}=\mcO(\sqrt{d})$. Given that $\kappa_{ef}$ and $\kappa_\sfH=\mcO(d\log d)$ are in our control (Appendix \ref{app:kappa-ef} and \ref{app:kappa-hessian}), we can choose $\mu$ such that $c_E>0$, i.e., $\mu^{-1}>(\sqrt{2\kappa_\sfH\kappa_{ef}}\land2\kappa_{ef})$. 
Nevertheless, the iteration complexity will be discussed for small enough $\epsilon$.
\end{remark} 

Relying on Lemma \ref{thm:delta_epsilon}, we now show the almost sure iteration complexity that is stronger than the claim that the random variable $\epsilon^2 T_\epsilon$ is tight or $\mcO_p(1)$.

\begin{theorem} \label{thm:almostsureiterationcomplexity}
    Let Assumptions \ref{assum:fn}-\ref{assum:lambda} hold. Then given $\epsilon>0$, $T_\epsilon\lesssim c_T\epsilon^{-2}$ almost surely 
    for some fixed $c_T > 0$. In particular, $\mbP\{\epsilon^2 T_\epsilon \leq c_T + \epsilon^2 K\}=1$ for $K$ being a positive integer valued random variable.
\end{theorem}
\begin{proof}
    Let $f^* := \min_{\BFx \in \mbR^d} f(\BFx) > -\infty$ be the optimal function value and $K_E(\omega)$ be the one defined in Lemma~\ref{thm:delta_epsilon}. We can find from Theorem \ref{thm:almostsureconvergence} $K_h(\omega)$ such that for all $k \ge K_h(\omega)$ implies that $f(\BFX_{k}(\omega))< f^*_{st}+1$, where $f^*_{st}$ is the highest function value among the stationary points. Without loss of generality, given $\omega\in\Omega$, let $\epsilon_0(\omega)\in(0,1)$ be small enough such that except for a set of probability 0, the set $\mcK_{lb}(\omega)=\{K'_E(\omega)\leq k<T_\epsilon(\omega): \Delta_k(\omega)<c_{lb}\epsilon\}$ is nonempty for all $\epsilon\leq\epsilon_0(\omega)$, where $c_{lb}$ is the one defined in Lemma~\ref{thm:delta_epsilon} and $K'_E(\omega):=\max\{K_h(\omega),K_E(\omega)\}$ . This implies that $\Delta_k(\omega)\geq\gamma_2c_{lb}\epsilon$ for all $K'_E(\omega) \le k < T_\epsilon(\omega)$. For the remainder of the proof we will use the notation $f_k(\omega):=f(\BFX_k(\omega))$ for simplicity. Following the steps in the proof of Theorem \ref{thm:deltaconverge} we have
    \begin{align*}
        \theta'\sum_{\substack{k\in\mcK \\ k\ge K'_E(\omega) }}\Delta_k^2(\omega) 
        &\le f_{K'_E(\omega)}(\omega) - f^* +\sum_{k=K'_E(\omega)}^{\infty}(|\Ebar_{k}(\omega)|+|\Ebar_{k+1}(\omega)|),
    \end{align*}
    from which we obtain 
    \begin{align*}
    \sum_{k=K'_E(\omega)}^{\infty}\Delta_k^2(\omega) \le \frac{\gamma_1^2}{1-\gamma_2^2}\sum_{i=0}^{\infty}\Delta_{k_i}^2(\omega)< \frac{\gamma_1^2}{1-\gamma_2^2}\left(\Delta_{\max}^2+\frac{f_{K'_E(\omega)}(\omega)-f^*+E'_{K'_E(\omega),\infty}(\omega)}{\theta'}\right),
    \end{align*}
    where $E'_{i,j}(\omega) = \sum_{k=i}^{j}(|\Ebar_{k}(\omega)|+|\Ebar_{k+1}(\omega)|)$ and $k_i$ is the $i$-th successful iteration after $K'_E(\omega)$.
    Then, by the definition of $K'_E(\omega)$, we have $|\Ebar_k(\omega)|\leq c_E\Delta_k^2(\omega)$ for any $k \ge K'_E(\omega)$, which implies $E'_{K'_E(\omega),\infty}(\omega) \le 2c_E\sum_{k=K'_E(\omega)}^{\infty}\Delta_k^2(\omega)$. As a result, we obtain with small enough $c_E$,
    \begin{equation}
    \label{eq:delta-finite}
        \left(1-\frac{2\gamma_1^2 c_E}{(1-\gamma_2^2)\theta'}\right)\sum_{k=K'_E(\omega)}^{\infty}\Delta_k^2(\omega) < \frac{\gamma_1^2}{1-\gamma_2^2}\left(\Delta_{\max}^2+\frac{f^*_{st}-f^*+1}{\theta'}\right).
    \end{equation}
    Now we can write $\sum_{k=K'_E(\omega)}^{\infty}\Delta_k^2(\omega)>\sum_{k=K'_E(\omega)}^{T_\epsilon(\omega)-1}\Delta_k^2(\omega)>\gamma_2^2c_{lb}^2\epsilon^2(T_\epsilon(\omega)-K'_E(\omega)).$
    Then by \eqref{eq:delta-finite}, we can have $\epsilon^2T_\epsilon(\omega)<c_T + \epsilon^2K'_E(\omega)$ for all $\epsilon\leq\epsilon_0(\omega)$.
\end{proof} Note that Theorem~\ref{thm:almostsureiterationcomplexity} implies that $\limsup_{\epsilon\to0}\epsilon^2T_\epsilon(\omega)\leq c_T$ for all $\omega\in\Omega$, that is, it is bounded by a fixed value in the limit, almost surely. 

\begin{remark}
    A closer look in the proof of Theorem~\ref{thm:almostsureiterationcomplexity} suggests that Algorithm \ref{alg:ASTRO-DF-refined} has better iteration complexity than the original ASTRO-DF in the constant terms. This is because (dropping $\omega$ for ease of exposition) we have $\Delta_k \geq \gamma_2 c_{lb}\epsilon$ for all $K_E\le k<T_\epsilon$. If for all $k<K_E$, $\exists c'_{lb}>0$ such that $\Delta_k\geq\gamma_2 c'_{lb}\epsilon$, then $\sum_{k=0}^{T_\epsilon-1}{\Delta_k}^2 > \gamma_2^2 {c_{lb}}^2 \epsilon^2 (T_\epsilon - K_E)+{\gamma_2}^2 {c'_{lb}}^2 K_E \epsilon^2$. A larger value of $c'_{lb}$ leads to a larger lower bound on the trust-region size, resulting in a smaller sample size $(N_k)$ and a larger step size. This explains why the probability of having a successful iteration $k$ is paramount. The two refinements play a crucial role in increasing $c_{lb}$, as stated in Remark \ref{remark:delta_epsilon} and Theorem \ref{thm:increaseing-prob}. 
    Additionally, under certain regularity conditions of random variable $K_E$ such as finite first moment, Theorem~\ref{thm:almostsureiterationcomplexity} becomes equivalent to the $L_1$ results or $\mbE[T_\epsilon]=\mcO(\epsilon^{-2})$ that is similar to what \cite{blanchet2019convergence} achieve. However, we attain this canonical rate without assumptions of probabilitally fully linear models, independence, or using renewal theory.
\end{remark}

\section{Numerical Results}
\label{sec:numerical}
In this section, we evaluate and compare simulation optimization solvers on problems from the SimOpt library \citep{eckman2022simopt}. 

The SimOpt library includes the MRG32k3a~\citep{LEcuyer2002} pseudorandom-number generator and common random numbers for all solvers to manage uncertainties during search and evaluation and enable efficient comparisons. The SimOpt solver library includes various solvers like Nelder-Mead, Random Search, ALOE \citep{jin2021probability}, ADAM \citep{kingma2017adam}, STRONG \citep{Chang2013STRONG}, and STORM \citep{chen2018storm}. The SimOpt problem library consists of optimization problems where the simulation oracle provides the objective function value at specific points. Due to limited information on the objective function's structure, stochastic simulation oracles are preferred over deterministic problems with \emph{added} stochastic noise, as the latter leads to artificial solution-dependent estimators. The evaluation of solvers in the SimOpt library involves two main procedures. First, we run $m$ macroreplications for each solver and problem. The solver aims to solve the problem during each macroreplication until a pre-defined budget is exhausted. At each solution, $\BFx$, the objective function is estimated by conducting $n$ replications using sample average approximation, which varies depending on the solver used (adaptive solvers use a random sample size $N(\BFx)$). Second, we conduct $\ell$ post-replications at the intermediate solutions of each macroreplication to estimate the objective function without optimization bias. In our experiments, we test the performance of the solvers using $m=20$ macroreplications and $\ell=200$ post-replications.

We use the following standard parameters: $\mu = 1000, \eta = 0.5, \gamma_1 = 0.75,$ and $\gamma_2 = 1.5$. To determine $\Delta_{\max}$ for each macroreplication, we employ a process that generates random solutions for the problem of interest, and the maximum distance between them is calculated and set as $\Delta_{\max}$. For each sample path, we tune $\Delta_0$ by a pilot run as one of three possibilities, $0.05\Delta_{\max}\times(0.1,1,10)$ using $1\%$ of the total budget for each. We also tune the value of the scaling parameter $\kappa$ at the first iteration for each sample path by setting $\kappa=\Fbar(\BFX_0,N_0)/\Delta_0^2$. Hence, $\kappa$ also has three possibilities based on $\Delta_0$. This tuning approach enables us to adjust the scaling of $\Delta_0, \Delta_{\max},$ and $\kappa$ in response to the behavior of the optimization algorithm. The original ASTRO-DF algorithm utilized local models with linear interpolation and implemented a strategy of reusing design points from previous iterations by following the \textit{AffPoints} algorithm presented in \cite{wild2008orbit}, enabling the reuse of design points as extensively as possible. For details of the algorithms compared, see \cite{simoptgithub}.

As presented in Section~\ref{sec:contribution}, Figure \ref{fig:entire} displays the solvability profiles for 60 problems from the SimOpt library. A solvability profile of a solver depicts the proportion of tested problems solved within a certain relative optimality gap. The refined ASTRO-DF solver solved more than 80\% of the problems within 30\% of the budget, significantly outperforming the contenders. Next, we will examine each refinement and its corresponding effect. 

\paragraph{Effect of Diagonal Hessians}
\cite{ha2021improved} compare (i) ASTRO-DF with full Hessian, (ii) ASTRO-DF with diagonal Hessian following Definition \ref{defn:diagHess}, and (iii) ASTRO-DF that integrates both linear and fully quadratic models through a heuristic approach that utilizes linear models when far from first-order optimality and quadratic models otherwise. Experimenting with these three versions on three problems from the SimOpt library indicated that the diagonal Hessian version was capable of the fastest progress with robustness (lower variance). In the remainder of this section, we include a 20-dimensional problem to further investigate the algorithm behavior in higher dimensions.

\paragraph{Effect of Direct Search}
We investigate the effect of using a direct search within ASTRO-DF after having implemented the first refinement that yields diagonal Hessian quadratic models using coordinate basis placements of interpolation points. In addition to a broad comparison of the new solver on 60 problems as illustrated in Figure~\ref{fig:entire}, we conduct experiments for two problems, namely, the Stochastic Activity Network (SAN) problem, which is a convex 13-dimensional problem, and a 20-dimensional Rosenbrock function with multiplicative noise, $F(\BFx,\xi) = \sum_{i=1}^{19}\left[100\left(x_{i+1} - \xi_i x_i^2\right)^2+(\xi_i x_i-1)^2\right]$, 
where $\xi_i~\mcN(1,0.1)$ for all $i \in \{1,\dots, 19\}$ \citep{kim2010rsbr}. The reason why this function has become a popular choice for evaluating optimization algorithms is attributable, in part, to the fact that its global minimum is located within a long, narrow valley that displays a parabolic shape (highly nonconvex). This characteristic makes the problem particularly challenging. 

Along with Figure \ref{fig:delta}, Figure \ref{fig:fbars} now shows the solver's finite time performance in terms of the progress made per iteration with mean and 95\% confidence interval after running the algorithms 20 times. In SimOpt and typical simulation experiments, a fixed total simulation budget is given; hence, the number of iterations completed varies from run to run and solver to solver. For our purposes here, we display the first 100 iterations for the two versions of ASTRO-DF. For both problems, ASTRO-DF with direct search demonstrates a noticeably faster rate of progress in the initial 30 iterations of the algorithm while revealing a nearly identical rate afterward. This better early progress can be explained by more efficient utilization of the limited budget, resulting from a higher likelihood of successful iterations.

\begin{figure} [htp]
\centering
\subfloat[20-dim noisy Rosenbrock]{%
\resizebox*{6.5cm}{!}{\includegraphics{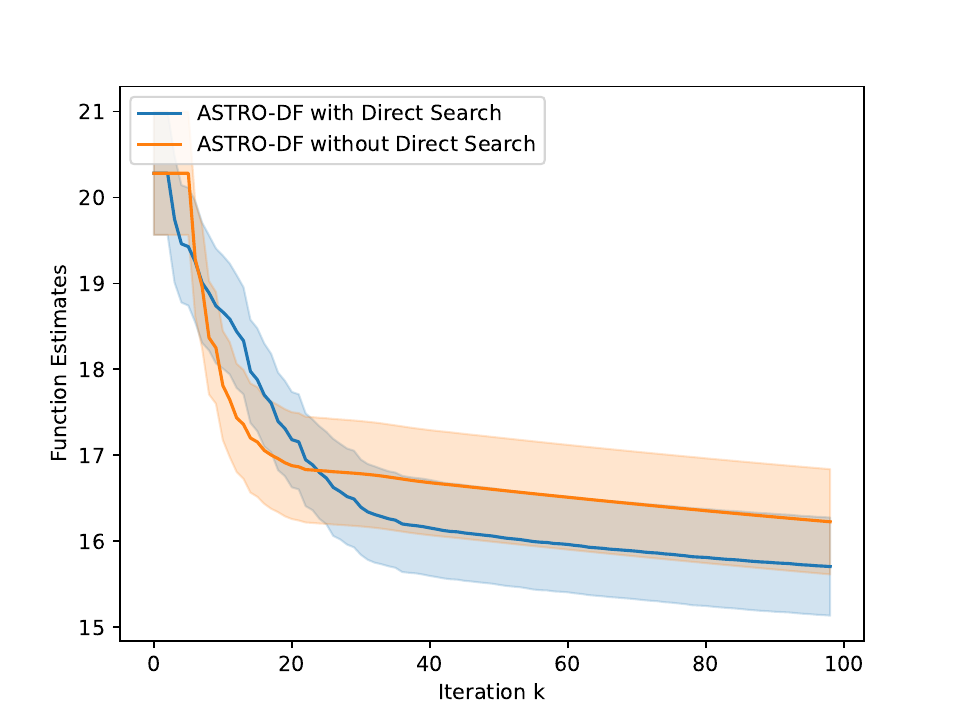}}\label{fig:fbar1}}\hspace{5pt}
\subfloat[13-dim Stochastic Activity Network]{%
\resizebox*{6.5cm}{!}{\includegraphics{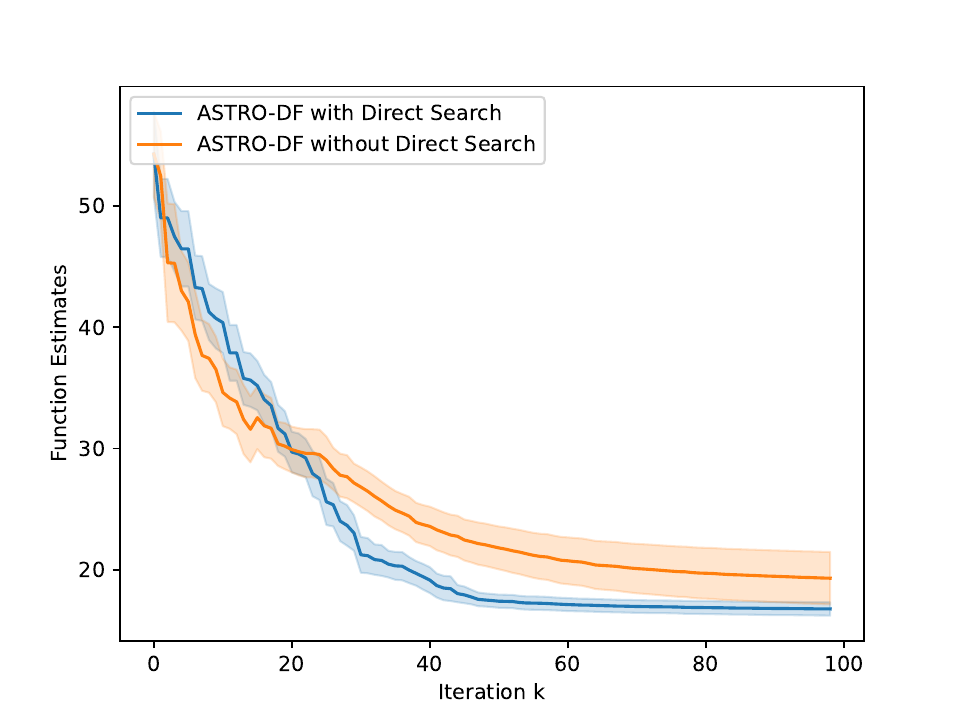}}\label{fig:fbar2}}
\caption{Better progress of ASTRO-DF in the early iterations with direct search is evident in the mean function estimates and 95\% confidence intervals on both problems.} \label{fig:fbars}
\end{figure}

To better map the improvements in gains per iteration to gains per simulation runs, we also report Figure~\ref{fig:totalwork_san} for the SAN problem. This figure reveals that ASTRO-DF without direct search can only have about 40 iterations with a 30,000 budget, whereas the ASTRO-DF with direct search reaches 100 iterations with the same budget. The evident cause of less progress in the original ASTRO-DF is the rapid reduction of $\Delta_k$ to find new incumbents with a satisfactory reduction, necessitating an additional budget for each iteration -- a faster increase in the expended budget after iteration 20.

\begin{figure} [htp]
\centering
\includegraphics[width=0.5\columnwidth]{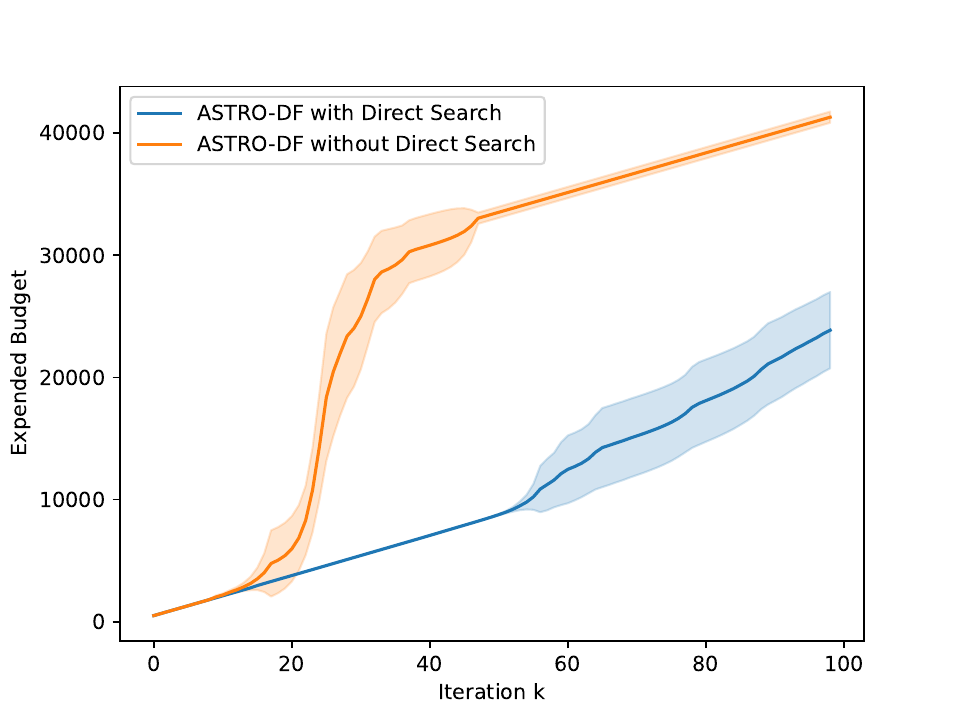}
\caption{The total budget spent per iteration is less due to savings from using direct search, as tested in SAN with 95\% confidence intervals from 20 macroreplications.} \label{fig:totalwork_san}
\end{figure}

Lastly, Figure~\ref{fig:objective} shows the mean progress with 95\% confidence for each problem as a function of the expended budget. In both cases, significantly better solutions are reached with high confidence due to the boost in the early iterations of the algorithm, as discussed above. Especially for the SAN problem, we observe that direct search remarkably accelerates the convergence in the first 30\% of the expended budget with less uncertainty (narrower confidence intervals). In the Rosenbrock function, the improvement, while significant, is not as pronounced. We attribute consistently better solutions with direct search in this non-convex noisy problem to a better exploration of the feasible region (more chances of success means visiting more points) without an additional cost, besides the saving of budget and a slower rate of $\Delta_k$ decay.

\begin{figure} [htp]
\centering
\subfloat[20-dim noisy Rosenbrock]{%
\resizebox*{6.5cm}{!}{\includegraphics{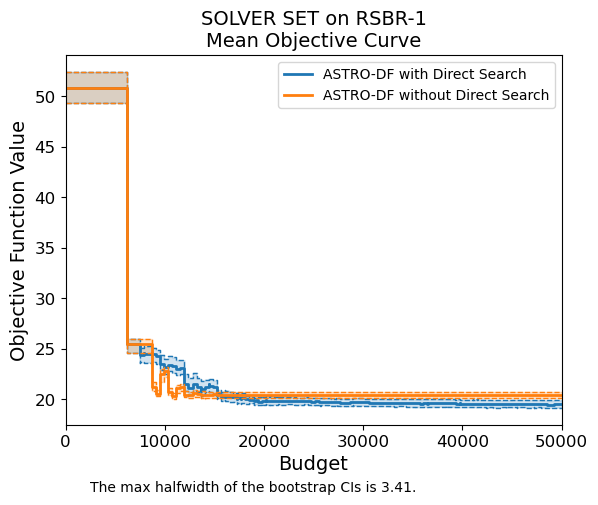}}\label{fig:objective_rosen}}\hspace{5pt}
\subfloat[13-dim Stochastic Activity Network]{%
\resizebox*{6.5cm}{!}{\includegraphics{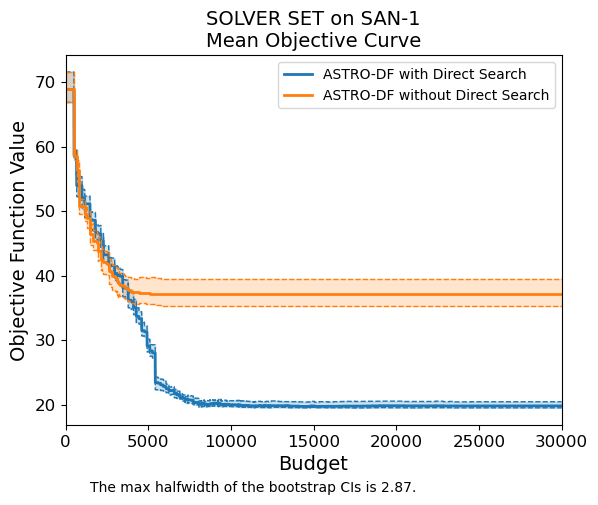}}\label{fig:objective_san}}
\caption{Mean value and 95\% confidence interval of objective function trajectory per spent budget for both problems exhibits significant finite-time improvement with direct search.} \label{fig:objective}
\end{figure}

\section{Conclusion and Future Work} \label{sec:conclusion}
This paper proposes efficient procedures and theoretical advancements for solving SDFO to first-order optimality in finite time. Trust regions are a class of algorithms with growing popularity for their volatile mechanics and stability in this context. However, expensive function approximations with locally certified models have caused a shortage of implementable trust-region algorithms for SDFO that serve higher dimensions robustly and efficiently. Our approach of integrating adaptive sampling strategies with local approximations via coordinate-basis designs and leveraging intermediate direct search steps addresses this limitation. The resulting algorithm is implemented and experimented on a testbed of SDFO problems and maintains almost sure convergence guarantees. Crucially, we also prove canonical complexity rates in almost sure sense and in expectation without requiring practically hard assumptions. The gain in finite-time performance is apparent with justifiable faster progress in the early iterations of the search.

A future research direction is on the exploration of the work complexity or total oracle runs and the dimension-dependent constants. Additionally, reusing history and allowing the design set geometry to vary may be beneficial. Although the refined ASTRO-DF effectively utilizes information by incorporating direct search, the data (visited points) from previous iterations is not carried over to subsequent iterations. Therefore, there might be potential to enhance efficiency by systematically reusing the design points and replications while retaining the benefits of ASTRO-DF with coordinate direct search. Other research areas include handling higher dimensions with techniques like random subspaces, in which one generates a sequence of solutions to random embedding constructions in lower dimensions without losing much information.\citep{cartis2022randomsubspace, dzahini2022randomsubsapce}. 

\if0\blind{
\section*{Acknowledgements}
The authors acknowledge the helpful discussions with Dr. Raghu Pasupathy that led to the improvement of the manuscript and the generous support from the National Science Foundation Grant CMMI-2226347.	} \fi

\bibliographystyle{apalike}
\spacingset{1}
\bibliography{main-paper}

\newpage
\appendix
\section{Upper Bound for the Model Error}
\label{app:kappa-ef}
We show that the upper bound for the model  error is $\mcO(\Delta_k^2 + E)$, where $E$ denotes the stochastic error. This finding implies that, from Theorem \ref{thm:bddE}, the fully-linear model can be attained almost surely when $k$ is sufficiently large, provided that the sampling rule \eqref{eq:inner-sampling-interpolation} is followed.

\begin{theorem}
    Let Assumption \ref{assum:fn} and \ref{assum:error} hold, and the interpolation model $M_k(\BFx)$ of $f$ be a stochastic quadratic models with diagonal Hessian constructed using $\Fbar(\BFX_k^{i},n(\BFX_k^{i}))=f(\BFX_k^{i}) + \Ebar_k^{i},$ for $i=0,1,\dots,2d$. Then, we can uniformly bound the model  error by, for any $\BFx\in\mcB\left(\BFX_k;\Delta_k\right)$,
    \begin{align*}
        | M_k(\BFx) - f(\BFx) | \le \kappa_{ef} \Delta_k^2 + |\Ebar_k^{0}| + \sum_{i=1}^{2d}\left|\Ebar_k^{i}-\Ebar_k^{0}\right|,
    \end{align*}
    where $\kappa_{ef} = \kappa_{eg1}+\frac{(\kappa_{Lg}+\kappa_{\sfH})}{2}$ and $\kappa_{eg1}$ = $\frac{5\sqrt{2d}}{2} (\kappa_{Lg} + \kappa_{\sfH})$. If in addition $n(\BFX_k^{i})=N_k^i$ following the adaptive sampling rule~\eqref{eq:inner-sampling-interpolation}, then $| M_k(\BFx) - f(\BFx) | \le (\kappa_{ef}+(4d+1)c_E)\Delta_k^2$ for a given constant $c_E>0$ and a sufficiently large $k$ with probability 1.
\end{theorem}
\begin{proof}
    Based on Taylor expansion, we know 
    \begin{align}
    \label{eq:functionTaylor}
        f(\BFx)-f(\BFX_k^{0}) \le (\BFx-\BFX_k^{0})^\intercal \TD f(\BFX_k^{0}) + \frac{\kappa_{Lg}}{2} \|\BFx - \BFX_k^{0}\|^2.
    \end{align}
    Then we obtain by subtracting \eqref{eq:functionTaylor} from \eqref{eq:mdefn}
    \begin{align*}
        (f(\BFx)-f(\BFX_k^{0})) - (M_k(\BFx)-\beta_0) \le (\BFx-&\BFX_k^{0})^\intercal (\TD f(\BFX_k^{0}) - \BFG) \\&
        + \frac{\kappa_{Lg}}{2} \Delta_k^2-\frac{1}{2}(\BFx-\BFX_k^{0})^\intercal \sfH (\BFx-\BFX_k^{0}).
    \end{align*}
    Since we know $\beta_0 = \Fbar(\BFX_k^{0}, n(\BFX_k^{0}))=f(\BFX_k^{0}) + \Ebar_k^{0}$ for a stochastic quadratic models with diagonal Hessian, we have 
    \begin{align*}
        |f(\BFx) - M_k(\BFx)| &\le  \|\BFx-\BFX_k^{0}\| \|\TD f(\BFX_k^{0}) - \BFG\| + \frac{\kappa_{Lg}}{2} \Delta_k^2 + \frac{1}{2}\|\sfH\| \Delta_k^2+|\Ebar_k^{0}|\\
        &\le \Delta_k \|\TD f(\BFX_k^{0}) - \BFG\| + \frac{\kappa_{Lg}+\kappa_{\sfH}}{2} \Delta_k^2 +|\Ebar_k^{0}|\\
        &\le \left(\kappa_{eg1}\Delta_k + \frac{\sqrt{\sum_{i=1}^{2d}\left(\Ebar_k^{i}-\Ebar_k^{0}\right)^2}}{\Delta_k}\right)\Delta_k + \frac{\kappa_{Lg}+\kappa_{\sfH}}{2} \Delta_k^2 +|\Ebar_k^{0}|\\
        &\leq \left(\kappa_{eg1}+\frac{(\kappa_{Lg}+\kappa_{\sfH})}{2}\right) \Delta_k^2 + |\Ebar_k^{0}| + \sum_{i=1}^{2d}\left|\Ebar_k^{i}-\Ebar_k^{0}\right|.
    \end{align*} 
    With adaptive sample sizes $N_k^i$ for $i=0,1,\dots,2d$, given a constant $c_E>0$ and sufficiently large $k$, the error terms can be bounded by $|\Ebar_k^{0}|\leq c_E\Delta_k^2$ and $\left|\Ebar_k^{i}-\Ebar_k^{0}\right|\leq2c_E\Delta_k^2$ with probability 1. This in turn completes the proof.
\end{proof}

\section{Upper Bound for $\|\sfH_k\|$}
\label{app:kappa-hessian}

We show that the model Hessian can be bounded by $\mcO(\|\TD f(\BFX_k)\|\Delta_k^{-1} + E\Delta_k^{-2})$, where $E$ represents the stochastic error. Theorem \ref{thm:delta_epsilon} indicates that $\Delta_k > \gamma_2 c_{lb}\epsilon$ for any $k<T_\epsilon$, provided that $c_{lb}$ is sufficiently small. So given $\epsilon >0$, $\|\TD f(\BFX_k)\|\Delta_k^{-1}$ can also be bounded by a fixed constant. Additionally, we found from Theorem \ref{thm:bddE} that $E\Delta_k^{-2}$ is also bounded by another fixed constant. Therefore, there exists some $\kappa_{\sfH} > 0$ such that $\|\sfH_k\| \le \kappa_{\sfH}$ for all $k$ with probability 1. It is worth noting that we still need Assumption \ref{assum:fcd} because it is needed to prove Theorem \ref{thm:delta_epsilon}.

\begin{theorem}
    Let Assumption \ref{assum:fn} and \ref{assum:error} hold, and the interpolation model $M_k(\BFx)$ of $f$ be a stochastic quadratic models with diagonal Hessian constructed using $\Fbar(\BFX_k^{i},n(\BFX_k^{i}))=f(\BFX_k^{i}) + \Ebar_k^{i},$ for $i=0,1,\dots,2d$. Then, we can bound $\|\sfH_k\|$ by
    \begin{align*}
        \|\sfH_k\| \le \max_{i\in\{1,2,\dots,d\}} \left(\kappa_{Lg} + \frac{2\|\TD f(\BFX_k)\|}{\Delta_k} + \frac{|\Ebar_k^{i}|+|\Ebar_k^{i+d}|+|2\Ebar_k^{0}|}{\Delta_k^2}\right).
    \end{align*} If in addition $n(\BFX_k^{i})=N_k^i$ following the adaptive sampling rule~\eqref{eq:inner-sampling-interpolation}, then $\|\sfH_k\|$ is bounded by a constant $\kappa_\sfH$ for large enough $k$ almost surely.
\end{theorem}
\begin{proof}
    Let $[\sfH_k]_i$ be the $i$-th diagonal element of $\sfH_k$. We already know that
    \begin{equation*}
        [\sfH_k]_i = \frac{f(\BFX_k+\BFe_i\Delta_k) + f(\BFX_k-\BFe_i\Delta_k)-2f(\BFX_k)+\Ebar_k^{i}+\Ebar_k^{i+d}-2\Ebar_k^{0}}{\Delta_k^2}.
    \end{equation*}
    In addition, we know from the lipschitz continuity of the gradient,
    \begin{align*}
        |f(\BFX_k+\BFe_i\Delta_k)-f(\BFX_k)| &\le |\TD f(\BFX_k)^\intercal(\BFX_k+\BFe_i\Delta_k-\BFX_k)| + \frac{\kappa_{Lg}}{2}\|\BFX_k+\BFe_i\Delta_k-\BFX_k\|^2 \\
        &\le \Delta_k \|\TD f(\BFX_k)\| + \frac{\kappa_{Lg}}{2}\Delta_k^2.
    \end{align*}
    Then, we have 
    \begin{align*}
        |[\sfH_k]_i| &\le \frac{ 2 \Delta_k \|\TD f(\BFX_k)\| + \kappa_{Lg}\Delta_k^2 + |\Ebar_k^{i}|+|\Ebar_k^{i+d}|+|2\Ebar_k^{0}|}{\Delta_k^2}\\
        &= \kappa_{Lg} + \frac{2\|\TD f(\BFX_k)\|}{\Delta_k} + \frac{|\Ebar_k^{i}|+|\Ebar_k^{i+d}|+|2\Ebar_k^{0}|}{\Delta_k^2}.        
    \end{align*}
    As a result, we have from $\|\sfH_k\|\le \max_{i\in\{1,2,\dots,d\}} |[\sfH_k]_i|$
    \begin{equation*}
        \|\sfH_k\| \le \max_{i\in\{1,2,\dots,d\}} \left(\kappa_{Lg} + \frac{2\|\TD f(\BFX_k)\|}{\Delta_k} + \frac{|\Ebar_k^{i}|+|\Ebar_k^{i+d}|+|2\Ebar_k^{0}|}{\Delta_k^2}\right).
    \end{equation*}
    Note, using the adaptive sample sizes $N_k^i$ for all $i=0,1,\ldots, 2d$ renders $|\Ebar_k^{i}|$, $|\Ebar_k^{i+d}|$, and $|\Ebar_k^{0}|$ all of the same order as $\Delta_k^2$ for large enough $k$, as $\|\TD f(\BFX_k)\|$ and $\Delta_k$ become of the same order too for large $k$. This yields a constant bound for large $k$ almost surely that due to the max operator appears of the magnitude $\mcO(d\log d)$.
\end{proof}
 \end{document}